\documentclass[11pt]{article}
\usepackage{amsmath,enumerate,amsfonts,amssymb,color,graphicx,amsthm,mathrsfs}
\usepackage{extarrows}
\allowdisplaybreaks[4]
\usepackage{authblk}
\usepackage{hyperref}
\usepackage[margin=1.25in]{geometry}
\usepackage{cite}
\hypersetup{
	colorlinks,
	citecolor=blue,
	filecolor=blue,
	linkcolor=blue,
	urlcolor=black
}

\numberwithin{equation}{section}
\newtheorem{thm}{Theorem}[section]

\newcommand{\R}{\mathbb{R}}

\newcommand{\be}{\begin{equation}}
	\newcommand{\ee}{\end{equation}}

\newtheorem{definition}{Definition}[section]
\newtheorem{lemma}{Lemma}[section]

\newtheorem{proposition}{Proposition}[section]
\newtheorem{corollary}{Corollary}[section]

\newcommand{\Ha}{\mathbb{H}^n}

\newsavebox{\Citename}

\numberwithin{equation}{section}



\begin{document}

	\title{\bf A Sharp Sobolev Inequality on Compact CR Manifolds}
	
	\author{ \sc Zhongwei Tang\thanks{Z. Tang is supported by National Natural Science Foundation of China (12471099,12271223).} ,\, Bingwei Zhang\\
		{\small School of Mathematical Sciences, Laboratory of Mathematics and Complex Systems, MOE\\
			\small	Beijing Normal University, Beijing 100875, People’s Republic of China	} }
	
	\date{}
	
	\maketitle
	
	\begin{abstract}
	For any \(n\geq 1\), let \((M,\theta)\) be a compact strictly pseudoconvex  hypersurface type CR manifold with dimension  \(2n+1\),  \(Q=2n+2\) be its homogeneous dimension. We prove the sharp Folland--Stein Sobolev inequality on \(M\): if \(Q^*=2Q/(Q-2)\), then there exists a constant \(A>0\), depending only on \((M,\theta)\), such that
		\[
		\|u\|_{L^{Q^*}(M)}^2 \le K(n)^2\|\nabla_b u\|_{L^2(M)}^2 + A\|u\|_{L^2(M)}^2, \qquad u\in S^{1,2}(M).
		\]
		Here \(K(n)\) is the sharp constant in the Folland--Stein Sobolev inequality on the Heisenberg group \(\mathbb H^n\).
	\end{abstract}
	
	{\noindent \small \bf Key words:} Sharp Sobolev inequality, compact CR manifolds,  Subelliptic equations,  Blow up analysis.
	
	{\noindent \small \bf Mathematics Subject Classification (2020)}\quad 32V20 · 35H20 ·  46E35
	\section{Introduction}
	
	Sharp Sobolev inequalities are among the basic tools in the study of partial differential equations, especially those arising from geometry and physics.  For \(N \ge 3\), Aubin \cite{A1976} and Talenti \cite{T1976} determined the optimal constant and all extremals in the sharp Sobolev inequality on \(\mathbb R^N\). More precisely, if \(2^*=2N/(N-2)\) and \(S_N\) denotes the optimal Sobolev constant, then
	\[
	S_N^{-2}=\inf \bigg\{\frac{\|\nabla u\|_{L^2(\mathbb{R}^N)}^2}{\|u\|_{L^{2^*}(\mathbb{R}^N)}^2}: u \in\  L^{2^*}(\mathbb{R}^N) \backslash\{0\},|\nabla u| \in L^2(\mathbb{R}^N)\bigg\}.
	\]
	Moreover, \(S_N^2=4 /[N(N-2) \sigma_N^{2 / N}]\), where \(\sigma_N\) is the volume of the standard \(N\)-sphere \(\mathbb{S}^N\). Up to multiplication by nonzero constants, all minimizers are given by
	\[
	U_{y, \lambda}(x)=\lambda^{(N-2) / 2} U(\lambda(x-y)),  \quad y \in \mathbb{R}^N, \quad \lambda>0,
	\]
	where
	\[
	U(x)=U_{0, 1}(x) =\Big(\frac{1}{1+a_N|x|^2}\Big)^{(N-2) / 2}, \quad a_N=\frac{1}{N(N-2) S_N^{2}}.
	\]
	With this choice of \(a_N\), the function \(U\) satisfies
	\[
	-\Delta U=S_N^{-2} U^{2^*-1} \quad \text { in } \mathbb{R}^N.
	\]
	With this normalization, \(U\) is the positive extremal satisfying
	\[
	U \in D^{1,2}(\mathbb{R}^N), \quad 0<U \leq 1, \quad U(0)=1, \quad \text{ and }\int_{\mathbb{R}^N} U^{2^*}\,d x=1.
	\]
	
	Aubin further studied the corresponding best-constant problem on compact Riemannian manifolds in \cite{A1976}. He proved that if \((M,g)\) has constant sectional curvature, then the Euclidean optimal constant remains valid up to a lower-order \(L^2\)-term: there exists a constant \(A>0\), depending only on \((M,g)\), such that
	\[
	\|u\|_{L^{2^*}(M,g)}^2 \le S_N^2\|\nabla_g u\|_{L^2(M,g)}^2 + A\|u\|_{L^2(M,g)}^2, \qquad u\in H^1(M).
	\]
	For arbitrary compact Riemannian manifolds, Aubin \cite{A1976} proved an almost sharp version in which \(S_N^2\) is replaced by \(S_N^2 + \varepsilon\), with the lower-order constant depending on \(\varepsilon\). He conjectured that the sharp inequality with the Euclidean constant \(S_N\) should hold on any compact Riemannian manifold; this conjecture was later proved by Hebey and Vaugon \cite{HV1996}. This best-constant viewpoint is part of the \(AB\) program in geometric analysis, whose purpose is to preserve the sharp model constant in the leading term and to understand the lower-order correction terms forced by the geometry; see Druet--Hebey \cite{DH2002}. The same best-constant problem was also studied by Hebey and Vaugon on complete, possibly noncompact Riemannian manifolds \cite{HV1995}. Related developments include sharp Sobolev inequalities under geometric assumptions on noncompact manifolds \cite{S1992,L1999,C1994}, general \(L^p\)-Sobolev and higher-order inequalities \cite{H1996,H1999,D1998}, and sharp trace inequalities on manifolds with boundary \cite{E1998,E1992,LZ1997,LZ1998}. In \cite{LR2003}, Li and Ricciardi obtained a refined sharp Sobolev inequality on compact Riemannian manifolds. More precisely, if \((M,g)\) is a smooth compact Riemannian manifold without boundary of dimension \(N\geq 6\), then there exists a constant \(A>0\), depending only on \((M,g)\), such that
	\[
	\|u\|_{L^{2^*}(M,g)}^2 \le S_N^2\int_M(|\nabla_g u|^2+c_N R_g u^2)\,dv_g + A\|u\|_{L^{2N/(N+2)}(M,g)}^2
	\]
	for all \(u\in H^1(M)\), where \(2^*=2N/(N-2)\), \(c_N=(N-2)/[4(N-1)]\), and \(R_g\) is the scalar curvature. They also proved sharpness results for the Euclidean constant, the scalar curvature term, and, in the non-locally conformally flat case, the exponent of the remainder term. More recently, the first author, together with Xiong and Zhou, revisited sharp Sobolev inequalities involving boundary terms and clarified the role of the boundary mean curvature in \cite{TXZ2021}. The first author and Zhou also studied a weighted Sobolev--Poincar\'e--type trace inequality on Riemannian manifolds in \cite{TZ2026}. These results provide further examples of how sharp Euclidean constants interact with boundary geometry and lower-order correction terms.
	
	The Heisenberg group, equipped with its standard contact structure and non-isotropic dilations, plays the role of the Euclidean model in Cauchy--Riemann(CR) geometry. In \cite{FS1974}, Folland and Stein established the Folland--Stein spaces on the Heisenberg group and on CR manifolds, which are the subelliptic counterparts of Sobolev spaces on Riemannian manifolds. They also proved the corresponding Sobolev embedding theorem in this setting. The natural dimension in this embedding is not the topological dimension \(2n+1\), but the homogeneous dimension \(Q=2n+2\). More precisely, if \(1<p<Q\) and \(p^*=pQ/(Q-p)\), then the Folland--Stein embedding takes the form
	\[
	S^{1,p}(M)\hookrightarrow L^{p^*}(M).
	\]
	
	On the Heisenberg group \(\mathbb{H}^n\), let \(K(n,p)\) denote the optimal constant in the inequality
	\[
	\Big(\int_{\mathbb H^n}|u|^{p^*}\,dv_0\Big)^{1/p^*} \leq K(n,p) \Big(\int_{\mathbb H^n}|\nabla_{\mathbb H}u|^p\,dv_0\Big)^{1/p}.
	\]
	For \(p=2\), the sharp form of this inequality was obtained by Jerison and Lee in their works on the CR Yamabe problem \cite{JL1987,JL1988,JL1989}. Thus, if \(Q^*=2Q/(Q-2)\), then \(K(n):=K(n,2)\) satisfies
	\[
	\Big(\int_{\mathbb H^n}|u|^{Q^*}\,dv_0\Big)^{2/Q^*} \leq K(n)^2\int_{\mathbb H^n}|\nabla_{\mathbb H}u|^2\,dv_0,
	\]
	for all functions \(u\) for which both sides are finite. Jerison and Lee also classified all extremals, which we will recall in the next section together with the preliminary material on pseudohermitian geometry and CR normal coordinates.

	Frank and Lieb \cite{FL2012} determined the sharp constants and classified all extremals for the Hardy--Littlewood--Sobolev inequalities on the Heisenberg group for \(0<\lambda<Q\), with the symmetric exponent \(r=2Q/(2Q-\lambda)\). By duality, their results yield sharp Sobolev inequalities for the sub-Laplacian and its conformally covariant fractional analogues. Their rearrangement-free approach establishes the existence of optimizers and classifies them through conformal normalization, second variation and spectral analysis on the CR sphere. In particular, they gives a direct proof of the Jerison--Lee inequality and its equality cases.

	Following the work of Frank and Lieb, Yan \cite{Y2023} established improved Sobolev inequalities for integer-order CR GJMS operators under higher-order moment constraints. His later work \cite{Y2025} treats both integer and fractional orders, combining improved inequalities with commutator identities to give alternative proofs of the existence and classification of extremals for the corresponding sharp inequalities obtained by Frank and Lieb \cite{FL2012}.

	Shi~\cite{S2018} extended the AB program to compact CR manifolds in the full Folland--Stein range. In particular, for every compact strictly pseudoconvex CR manifold \((M,\theta)\), every \(1<p<Q\), and every \(\varepsilon>0\), there exists \(A_\varepsilon>0\) such that
	\[
	\|u\|_{L^{p^*}(M)}^p \leq (K(n,p)^p+\varepsilon) \|\nabla_bu\|_{L^p(M)}^p + A_\varepsilon\|u\|_{L^p(M)}^p, \qquad u\in S^{1,p}(M).
	\]
	Equivalently, in the terminology of the AB program, the infimum of the leading constant on \(M\) is the Heisenberg sharp constant \(K(n,p)^p\). Specializing to \(p=2\), one obtains the almost sharp Folland--Stein inequality
	\begin{equation}\label{equShi}
		\|u\|_{L^{Q^*}(M)}^2 \leq (K(n)^2+\varepsilon)\|\nabla_bu\|_{L^2(M)}^2 + A_\varepsilon\|u\|_{L^2(M)}^2 .
	\end{equation}
	Related Aubin-type almost sharp inequalities for the principal parts of CR GJMS operators on compact strictly pseudoconvex CR manifolds were obtained by Yan \cite[Theorem~1.8]{Y2023}. This naturally leads to the sharp best-constant problem on compact CR manifolds: whether the small loss \(\varepsilon\) in the leading constant can be removed. Our main result gives an affirmative answer in the case \(p=2\), by proving the sharp compact inequality with the exact leading constant \(K(n)^2\).
	
	\begin{thm}\label{Mainthm}
		Let \((M, \theta)\) be a \(C^{\infty}\) compact strictly pseudoconvex  hypersurface type CR manifold with dimension \(2n+1\) for any \(n\geq 1\),  there exists a constant \(A>0\), depending only on \((M, \theta)\), such that for all \(\varphi \in S^{1, 2}(M, \theta)\), there holds
		\[
			\|\varphi\|_{L^{Q^*}(M)}^2 \leq K(n)^2\|\nabla_b \varphi\|_{L^2(M)}^2+A\|\varphi\|_{L^2(M)}^2.
		\]
	\end{thm}
    The proof is by contradiction and is motivated by Li--Ricciardi~\cite{LR2003}. Assuming that the desired sharp inequality fails, we construct a sequence of minimizers whose Sobolev quotients are strictly below the Heisenberg sharp threshold. We then perform a blow-up analysis and establish a local pointwise estimate by a Green-function comparison argument. Finally, we derive a contradiction from a localized energy estimate obtained by comparison with the sharp Folland--Stein inequality on \(\mathbb H^n\).	
	
	The paper is organized as follows. In Section~\ref{sectionpreliminary}, we recall the preliminary material on pseudohermitian geometry, Folland--Stein spaces, and CR normal coordinates. In Section~\ref{sectionblowup}, we construct the minimizing sequence and establish the basic blow-up estimates. In Section~\ref{sectionpointwise}, following the Green-function comparison argument of Li--Ricciardi \cite{LR2003}, we prove the local pointwise upper bound near the blow-up point. Finally, in Section~\ref{sectionpohozaev}, we obtain a localized energy estimate by applying the sharp Folland--Stein inequality on \(\mathbb H^n\) in pseudohermitian normal coordinates. The pointwise bound and subelliptic estimates allow us to control the resulting error terms, yielding the final contradiction and completing the proof of the main theorem.
	
	\section{Preliminaries}\label{sectionpreliminary}
	We first recall some basic notation and facts concerning pseudohermitian structures on CR manifolds. Our main references are Folland--Stein \cite{FS1974}, Webster \cite{W1978}, and Jerison--Lee \cite{JL1989}. We shall follow the notation used in \cite{W1978}. Unless otherwise specified, lower-case Greek indices range from \(1\) to \(n\), and the Einstein summation convention is in force. The Levi form is written as \(h_{\alpha\bar\beta}\). Its inverse is denoted by \(h^{\bar\beta\alpha}\), so that
	\[
	h_{\alpha\bar\gamma}h^{\bar\gamma\beta}=\delta_\alpha{}^\beta, \qquad h^{\bar\beta\gamma}h_{\gamma\bar\alpha}=\delta^{\bar\beta}{}_{\bar\alpha}.
	\]
	We use \(h_{\alpha\bar\beta}\) and \(h^{\bar\beta\alpha}\) to lower and raise indices. Complex conjugation is reflected in the indices; for instance,
	\[
	\bar{A}_{\alpha\bar{\beta} r}=A_{\bar{\alpha}\beta\bar{r}}.
	\]
	The pseudohermitian connection induces covariant differentiation on functions and tensors, which will be denoted by placing indices after a comma.
	
	Let \(M\) be a \(C^\infty\) smooth manifold of dimension \(m\), and let \(n\) be an integer with \(1\le n\le [m/2]\). A CR structure on \(M\) of complex dimension \(n\) is an \(n\)-dimensional smooth complex subbundle
	\[
	T_{1,0}\subset \mathbb{C}TM,
	\]
	called the holomorphic tangent bundle, satisfying
	\[
	T_{1,0}\cap T_{0,1}=\{0\}, \qquad [T_{1,0},T_{1,0}]\subset T_{1,0},
	\]
	where \(T_{0,1}:=\overline{T_{1,0}}\). Let \(G\) denote the real part of \(T_{1,0}+T_{0,1}\); equivalently, \(G\) is a real \(2n\)-dimensional subbundle of \(TM\) such that
	\[
	\mathbb{C}G=T_{1,0}+T_{0,1}.
	\]
	Moreover, there is a unique complex structure map \(J:G\to G\) satisfying
	\[
	J(V+\bar V)=i(V-\bar V), \qquad V\in T_{1,0}.
	\]
	Throughout this paper, we assume that the CR manifold \(M\) is orientable and of hypersurface type; namely, \(M\) is a \((2n+1)\)-dimensional orientable smooth manifold endowed with an \(n\)-dimensional CR structure.
	
	Let \(E=G^\perp\) denote the real line bundle in \(T^*M\) annihilating \(G\). Since \(M\) is assumed to be orientable and \(G\) is oriented by its complex structure, \(E\) admits a global nonvanishing smooth section. Let \(\theta\) be such a section of \(E\). The Levi form associated with \(\theta\) is the Hermitian form on \(T_{1,0}\) defined by
	\[
	L_\theta(V,\bar W) = \langle -2i\,d\theta, V\wedge \bar W\rangle, \qquad V,W\in T_{1,0}.
	\]
	We shall always assume that \(M\) is nondegenerate at every point; that is, if \(Z\in T_{1,0}\) satisfies
	\[
	L_\theta(Z,\bar W)=0 \qquad \text{for all } W\in T_{1,0},
	\]
	then \(Z=0\). Under this assumption, \(\theta\) is a contact form. If, in addition, \(L_\theta\) is positive definite for some choice of \(\theta\), then \(M\) is said to be strictly pseudoconvex. The distribution \(G=\ker\theta\) specifies the horizontal directions of differentiation, while strict pseudoconvexity ensures that the Levi form defines a positive definite metric on \(G\), which determines both the horizontal energy and the principal part of the sub-Laplacian.
	
	A contact form \(\theta\) on \(M\) determines a unique vector field \(T_\theta\), called the Reeb vector field, satisfying
	\[
	\theta(T_\theta)=1, \qquad T_\theta\rfloor d\theta=0.
	\]
	Since \(\theta\) is a contact form, commutators of horizontal vector fields generate the missing Reeb direction; hence the horizontal system satisfies the step-two H\"ormander bracket-generating condition. Let \(\{W_\alpha\}\) be a local frame for \(T_{1,0}\). The corresponding admissible coframe is the collection of \((1,0)\)-forms \(\{\theta^\beta\}\) defined by
	\[
	\theta^\beta(W_\alpha)=\delta_\alpha^\beta, \qquad \theta^\beta(W_{\bar\alpha})=\theta^\beta(T_\theta)=0.
	\]
	Then \(\{T_\theta,W_\alpha,W_{\bar\alpha}\}\) forms a local frame of \(\mathbb C TM\), and \(\{\theta,\theta^\alpha,\theta^{\bar\alpha}\}\) is the dual coframe. With respect to this coframe, one has
	\[
	d\theta= i h_{\alpha\bar\beta} \theta^\alpha\wedge\theta^{\bar\beta}, \qquad h_{\alpha\bar\beta} = L_\theta(W_\alpha,W_{\bar\beta}).
	\]

	The vector fields \(W_\alpha\) and \(W_{\bar\alpha}\) give the horizontal directions of differentiation, while \(h^{\bar\beta\alpha}\) is used to contract horizontal indices.
	
	The choice of the contact form \(\theta\) induces a natural linear connection, called the Tanaka--Webster connection, which was independently introduced by N. Tanaka \cite{T1975} and S. Webster \cite{W1978}. Following Webster's notation, this connection is defined by
	\[
	\nabla W_\alpha = {\omega_\alpha}^{\beta}\otimes W_\beta, \qquad \nabla W_{\bar\alpha} = {\omega_{\bar\alpha}}^{\bar\beta}\otimes W_{\bar\beta}, \qquad \nabla T_\theta=0, \]
	where the one-forms \({\omega_\alpha}^{\beta}\) are uniquely determined by
	\[
	d\theta^\alpha = \theta^\beta\wedge {\omega_\beta}^{\alpha} + \theta\wedge \tau^\alpha, \qquad \tau_\alpha\wedge \theta^\alpha=0, \qquad \omega_{\alpha\bar\beta} + \omega_{\bar\beta\alpha} = d h_{\alpha\bar\beta}.
	\]
	Writing \(\tau_\alpha=A_{\alpha\gamma}\theta^\gamma\), the second condition is equivalent to the symmetry relation \(A_{\alpha\gamma}=A_{\gamma\alpha}\).

	To express the connection in local coefficient form, set
\[
\Gamma_{\gamma\alpha}{}^\beta = \omega_\alpha{}^\beta(W_\gamma), \qquad \Gamma_{\bar\gamma\alpha}{}^\beta = \omega_\alpha{}^\beta(W_{\bar\gamma}), \qquad \Gamma_{0\alpha}{}^\beta = \omega_\alpha{}^\beta(T_\theta).
\]
Then
\[
\omega_\alpha{}^\beta = \Gamma_{\gamma\alpha}{}^\beta\theta^\gamma + \Gamma_{\bar\gamma\alpha}{}^\beta\theta^{\bar\gamma} + \Gamma_{0\alpha}{}^\beta\theta,
\]
or equivalently,
\[
\nabla_{W_\gamma}W_\alpha = \Gamma_{\gamma\alpha}{}^\beta W_\beta, \qquad \nabla_{W_{\bar\gamma}}W_\alpha = \Gamma_{\bar\gamma\alpha}{}^\beta W_\beta, \qquad \nabla_{T_\theta}W_\alpha = \Gamma_{0\alpha}{}^\beta W_\beta.
\]
The corresponding coefficients for \(W_{\bar\alpha}\) are obtained by complex conjugation. Thus these coefficients are not independent; they are the components of the uniquely determined Tanaka--Webster connection forms.
	
	The Tanaka--Webster connection induces covariant derivatives of functions and tensors. For a smooth function \(u\), we write
	\[
	u_\alpha=W_\alpha u,\qquad u_{\bar\alpha}=W_{\bar\alpha}u,\qquad u_0=T_\theta u.
	\]
	The Levi form gives a Hermitian metric on the horizontal bundle \(G=\ker\theta\). We denote by \(\nabla_b u\) the horizontal gradient of \(u\); in a local admissible frame,
	\[
	|\nabla_b u|^2 =h^{\alpha\bar\beta}u_\alpha u_{\bar\beta} +h^{\bar\alpha\beta}u_{\bar\alpha}u_\beta .
	\]
	The sub-Laplacian associated with \(\theta\) is the horizontal trace of the pseudohermitian Hessian. With the above connection forms,
	\[
	\Delta_b u =u_\alpha{}^\alpha+u_{\bar\alpha}{}^{\bar\alpha} =h^{\alpha\bar\beta}[ W_{\bar\beta}W_\alpha u+W_\alpha W_{\bar\beta}u -\Gamma_{\bar\beta\alpha}{}^\gamma W_\gamma u-\Gamma_{\alpha\bar\beta}{}^{\bar\gamma} W_{\bar\gamma}u ].
	\]
	Unless otherwise specified, all functions in this paper are real-valued. With our sign convention for \(\Delta_b\),
	\[
	\int_M \langle \nabla_bu,\nabla_bv\rangle_\theta\,dv_\theta = -\int_M v\,\Delta_bu\,dv_\theta, \qquad u,v\in C^\infty(M).
	\]
	In particular,
	\[
	\int_M |\nabla_bu|_\theta^2\,dv_\theta = -\int_M u\,\Delta_bu\,dv_\theta .
	\]
	The integration by parts formula stated above extends to \(S^{1,2}(M)\) in the usual weak sense.
	
	We next introduce the Heisenberg group, the flat model for strictly pseudoconvex pseudohermitian geometry. We identify \(\mathbb H^n\) with \(\mathbb C^n\times\mathbb R\), write its points as \(\xi=(z,t)\), and use the group law
	\[
	(z,t)\circ(z',t') =(z+z',\,t+t'+2\operatorname{Im}(z\cdot\overline{z'})), \qquad z\cdot\overline{z'}=\sum_{\alpha=1}^n z^\alpha\overline{z'^\alpha}.
	\]
	The standard left-invariant horizontal vector fields are
	\[
	Z_\alpha=\frac{\partial}{\partial z^\alpha} +i\bar z^\alpha\frac{\partial}{\partial t}, \qquad Z_{\bar\alpha}=\frac{\partial}{\partial\bar z^\alpha} -iz^\alpha\frac{\partial}{\partial t}, \qquad T=\frac{\partial}{\partial t}.
	\]
	They define the standard CR structure \(T_{1,0}(\mathbb H^n)=\operatorname{span}_{\mathbb C}\{Z_1,\ldots,Z_n\}\). The standard contact form is \(\Theta=dt+i z^\alpha d\bar z^\alpha-i\bar z^\alpha dz^\alpha\), so that \(d\Theta=2i\,dz^\alpha\wedge d\bar z^\alpha\). We denote by \(\nabla_{\mathbb H}\) and \(\Delta_{\mathbb H}\) the corresponding horizontal gradient and sub-Laplacian.
	
	The natural nonisotropic dilations on \(\mathbb H^n\) are \(\delta_\lambda(z,t)=(\lambda z,\lambda^2t)\), \(\lambda>0\). The homogeneous dimension is \(Q=2n+2\), and the Koranyi norm is
	\[
	\rho_0(\xi)=\rho_0(z,t)=(|z|^4+t^2)^{1/4}, \qquad \rho_0(\delta_\lambda\xi)=\lambda\rho_0(\xi).
	\]

	We now recall the pseudohermitian normal coordinates introduced by Jerison and Lee in their study of the CR Yamabe problem \cite{JL1989}.  In analogy with normal coordinates in the Riemannian setting, these coordinates identify a neighborhood of a point of \(M\) with a neighborhood of the origin in \(\mathbb H^n\).  We use only the pseudohermitian normal coordinates associated with the fixed contact form \(\theta\).  Their role in the present paper is to provide the coordinate maps used in the blow-up analysis and the local comparison between the pseudohermitian structure and the standard Heisenberg structure.

For a fixed point \((z,c)\in\mathbb H^n\), consider the curve \( \gamma_{z,c}(s):=\delta_s(z,c). \) With respect to the standard pseudohermitian connection on \(\mathbb H^n\), it satisfies
\[
\nabla_{\dot\gamma_{z,c}}\dot\gamma_{z,c}=2cT.
\]
Motivated by this model equation, for \(q\in M\) we using the splitting \( T_qM=G_q\oplus\mathbb RT_\theta(q), \) for \(V\in G_q\) and \(c\in\mathbb R\), let \(\gamma_{V,c}\) denote the solution of
\begin{equation}\label{eq:parabolic-geodesic}
\nabla_{\dot\gamma_{V,c}}\dot\gamma_{V,c} = 2cT_\theta(\gamma_{V,c}), \qquad \gamma_{V,c}(0)=q, \qquad \dot\gamma_{V,c}(0)=V.
\end{equation}
The curve \(\gamma_{V,c}\) is called the parabolic geodesic determined by \(V\) and \(c\).

As for any affine connection, the parabolic geodesic equation can be written as a system of ordinary differential equations in local coordinates. Let \(x=(x^1,\ldots,x^{2n+1})\) be local real coordinates, and write
\[
\nabla_{\partial_i}\partial_j = \Gamma_{ij}{}^k\partial_k, \qquad \partial_i=\frac{\partial}{\partial x^i},
\]
where \(1\leq i,j,k\leq 2n+1\) and \(\Gamma_{ij}{}^k\) are the coefficients of the Tanaka--Webster connection with respect to the coordinate frame. Writing
\[
T_\theta=T_\theta^k\partial_k,
\]
equation \eqref{eq:parabolic-geodesic} becomes
\[
\ddot{\gamma}^k(s) + \Gamma_{ij}{}^k(\gamma(s)) \dot{\gamma}^i(s)\dot{\gamma}^j(s) = 2cT_\theta^k(\gamma(s)).
\]
Thus the only difference from the usual coordinate geodesic equation is the forcing term \(2cT_\theta\) in the Reeb direction.

The associated parabolic exponential map is defined by \( \Phi_q(V+cT_\theta(q)):= \gamma_{V,c}(1). \) By \cite[Theorem~2.1]{JL1989}, \(\Phi_q\) maps a neighborhood of \(0\in T_qM\) diffeomorphically onto a neighborhood of \(q\).

Choose a frame \(\{e_\alpha\}_{\alpha=1}^n\) of \(T_{1,0}M|_q\) such that \( h_{\alpha\bar\beta}(q)=2\delta_{\alpha\bar\beta}. \) By \cite[Proposition~2.3]{JL1989}, this frame determines a smooth local frame \(\{W_\alpha\}\) of \(T_{1,0}M\), called the special frame, satisfying
\[
W_\alpha(q)=e_\alpha, \qquad h_{\alpha\bar\beta}=2\delta_{\alpha\bar\beta}
\]
throughout the coordinate neighborhood.  We denote its dual admissible coframe by \(\{\theta^\alpha\}\).

The coframe at \(q\) defines a real-linear identification \( \lambda_q:T_qM\longrightarrow\mathbb H^n \) by
\[
\lambda_q(V+cT_\theta(q))=(z,t), \qquad z^\alpha=\theta_q^\alpha(V), \qquad t=c,
\]
where \(V\in G_q\).  The pseudohermitian normal-coordinate parametrization centered at \(q\) is then \( \Psi_q:=\Phi_q\circ\lambda_q^{-1}. \)

Since \(M\) and its normalized frame bundle are compact, after shrinking the radius if necessary, there exists \(r_{JL}>0\), independent of \(q\) and of the normalized initial frame, such that \( \Psi_q:B_{r_{JL}}(0)\longrightarrow M \) is a diffeomorphism onto its image.  For \(0<r<r_{JL}\), we write \( B_r(q):=\Psi_q(B_r(0)), \) and define the corresponding Koranyi-type coordinate distance by \( \rho_q(x):=\rho_0(\Psi_q^{-1}(x)) \) for \(x\in B_{r_{JL}}(q)\).

We finally record only the asymptotic estimates from the Jerison--Lee Taylor expansions that will be used below.  After pulling all objects back by \(\Psi_q\), we suppress the pullback notation.  For \(1\le\alpha\le n\), one has
\begin{equation}\label{eq:JL-frame-comparison}
\begin{aligned}
W_\alpha &= Z_\alpha + \sum_{\beta=1}^n O(\rho_0^2)Z_\beta + \sum_{\beta=1}^n O(\rho_0^2)Z_{\bar\beta} + O(\rho_0^3)T .
\end{aligned}
\end{equation}
The conjugate expansion holds for \(W_{\bar\alpha}\).  After reducing \(r_{JL}\) if necessary, the inverse change of frame satisfies the analogous estimates.  Moreover,
\begin{equation}\label{eq:JL-volume-comparison}
\Psi_q^*dv_\theta = (1+O(\rho_0^2))\,dv_0 .
\end{equation}
Here every \(O(\rho_0^m)\) denotes a smooth coefficient function bounded in absolute value by \(C\rho_0^m\), where \(C\) may be chosen independently of \(q\) and of the normalized initial frame by the compactness of \(M\).

These estimates quantify how closely the pseudohermitian structure near \(q\) is approximated by the standard Heisenberg structure.  The coordinate map \(\Psi_q\) is used for the blow-up rescaling in Section~\ref{sectionblowup}, while \(\rho_q\) serves as the local radial variable in the Green-function comparison of Section~\ref{sectionpointwise}.  In Section~\ref{sectionpohozaev}, the frame and volume estimates above are used to compare the localized Sobolev quotient on \(M\) with the sharp Folland--Stein inequality on \(\mathbb H^n\).
	
	We now recall the Folland--Stein spaces, following \cite{FS1974,JL1987}. Let \(U\subset M\) be relatively compact and let \(\{W_1,\ldots,W_n\}\) be a local pseudohermitian frame on \(U\). Write \(X_j=\operatorname{Re}W_j\) and \(X_{j+n}=\operatorname{Im}W_j\), \(j=1,\ldots,n\). For a multi-index \(\alpha=(\alpha_1,\ldots,\alpha_m)\), with \(1\le\alpha_i\le2n\), set \(|\alpha|=m\) and \(X^\alpha=X_{\alpha_1}\cdots X_{\alpha_m}\). For \(1\le r\le\infty\) and \(k\ge0\), define
	\[
	\|u\|_{S^{k,r}(U)}=\sum_{|\alpha|\le k}\|X^\alpha u\|_{L^r(U)}.
	\]
	The space \(S^{k,r}(U)\) is the completion of \(C^\infty(U)\) with respect to this norm, equivalently the space of \(L^r\) functions whose horizontal derivatives \(X^\alpha u\), \(|\alpha|\le k\), belong to \(L^r\) in the distribution sense. On compact \(M\), \(S^{k,r}(M)\) is defined by a finite covering, a partition of unity, and the above local norms.
	
	We shall also use the Folland--Stein Lipschitz spaces. On \(\mathbb H^n\), for \(0<\beta<1\), \(\Gamma^\beta(\mathbb H^n)\) consists of bounded continuous functions \(u\) such that
	\[
	\|u\|_{\Gamma^\beta(\mathbb H^n)} =\|u\|_{L^\infty(\mathbb H^n)} +\sup_{\xi\in\mathbb H^n,\ \eta\ne0} \frac{|u(\xi\circ\eta)-u(\xi)|}{\rho_0(\eta)^\beta} <\infty.
	\]
	For \(\beta=1\), the first difference is replaced by the second difference
	\[
	\sup_{\xi\in\mathbb H^n,\ \eta\ne0} \frac{|u(\xi\circ\eta)+u(\xi\circ\eta^{-1})-2u(\xi)|}{\rho_0(\eta)}<\infty.
	\]
	Higher order spaces are defined by horizontal derivatives: if \(\beta=k+\beta'\), \(k\ge1\), \(0<\beta'\le1\), then \(\Gamma^\beta(\mathbb H^n)\) consists of those \(u\in L^\infty\cap C^0\) for which \(Du\in\Gamma^{\beta'}(\mathbb H^n)\) for every product \(D\) of at most \(k\) left-invariant horizontal vector fields. On a CR manifold, the local spaces \(\Gamma^\beta(U)\) are defined in normal coordinate neighborhoods by using the nonisotropic distance \(p(x,y)=\rho_0(\Psi_x^{-1}(y))\).
	
	The Folland--Stein Sobolev embedding gives, for \(Q=2n+2\) and
	\[
	Q^*=\frac{2Q}{Q-2}=2+\frac2n,
	\]
	the continuous embedding \(S^{1,2}(M)\hookrightarrow L^{Q^*}(M)\). In particular, there is a constant \(C\), depending only on \((M,\theta)\), such that
	\[
	\|u\|_{L^{Q^*}(M)}^2 \le C(\|\nabla_bu\|_{L^2(M)}^2+\|u\|_{L^2(M)}^2), \qquad u\in S^{1,2}(M).
	\]
	Moreover, \(S^{1,2}(M)\hookrightarrow L^p(M)\) compactly for every \(1\le p<Q^*\).
	
	Finally, we recall the sharp Folland--Stein inequality and its extremals on \(\mathbb H^n\). With respect to \(dv_0=\Theta\wedge(d\Theta)^n\), there is an optimal constant \(K(n)>0\) such that
	\[
	\Big(\int_{\mathbb H^n}|u|^{Q^*}\,dv_0\Big)^{2/Q^*} \le K(n)^2\int_{\mathbb H^n}|\nabla_{\mathbb H}u|^2\,dv_0.
	\]
	For the classification of all positive extremals, see Jerison--Lee \cite{JL1988} and Frank--Lieb \cite{FL2012}. We fix the normalized bubble
	\[
	v_0(z,t)= \Big(\frac{a_n^2}{(a_n+|z|^2)^2+t^2}\Big)^{\frac{Q-2}{4}}, \qquad a_n=\frac1\pi,
	\]
	so that
	\[
	v_0(0)=1,\qquad \nabla_{\mathbb H}v_0(0)=0,\qquad \int_{\mathbb H^n}v_0^{Q^*}\,dv_0=1.
	\]
	Every positive extremal normalized by \(\int_{\mathbb H^n}v^{Q^*}\,dv_0=1\) has the form
	\[
	v_{\lambda,\xi_0}(\xi) =\lambda^n v_0(\delta_\lambda(\xi_0^{-1}\circ\xi)), \qquad \lambda>0,\quad \xi_0\in\mathbb H^n.
	\]
	Equivalently, for \(\xi=(z,t)\) and \(\xi_0=(z_0,t_0)\),
	\[
	v_{\lambda,\xi_0}(z,t)= \Big( \frac{a_n^2\lambda^2} {(a_n+\lambda^2|z-z_0|^2)^2 +\lambda^4(t-t_0+2\operatorname{Im}(z\cdot\overline{z_0}))^2} \Big)^{\frac{Q-2}{4}}.
	\]
	Without the \(L^{Q^*}\)-normalization, one may also multiply these functions by a positive constant.

	\section{Some blow-up estimates}\label{sectionblowup}
	In this section, we provide some standard methods and estimates for our problem.
	
	Let \(M_1,\ldots,M_N\) be the finitely many connected components of \(M\). If Theorem~\ref{Mainthm} has been proved for connected compact CR manifolds, then for each \(M_j\) there exists a constant \(A_j>0\) such that the desired inequality holds on \(M_j\). Taking \(A=\max_{1\leq j\leq N}A_j\), the same inequality holds on \(M\). Therefore, in the sequel we assume without loss of generality that \(M\) is connected.
	
	For every \(\ell>0\), we define the functional \(I_\ell: S^{1, 2}(M) \rightarrow \mathbb{R}\) by
	
	\[
		I_\ell(u)=\frac{\|\nabla_b u\|_{L^2}^2+\ell\|u\|_{L^2}^2}{\|u\|_{L^{Q^*}}^2}.
	\]
	
	We argue by contradiction. Suppose that Theorem~\ref{Mainthm} fails. Equivalently, that for every \(\ell>0\), there exists some \(u \in S^{1, 2}(M)\) such that \(I_\ell(u)<K(n)^{-2}\). Indeed, if \(\lambda_{\ell_0}\ge K(n)^{-2}\) for some \(\ell_0>0\), then
	\[
	\|u\|_{L^{Q^*}(M)}^2 \le K(n)^2\|\nabla_bu\|_{L^2(M)}^2 + K(n)^2\ell_0\|u\|_{L^2(M)}^2,
	\]
	for all \(u\in S^{1,2}(M)\), which is exactly Theorem~\ref{Mainthm} with \(A=K(n)^2\ell_0\).
	
	\begin{proposition}
		For all \(\ell>0\), there exists a minimizer \(u_\ell \in S^{1, 2}(M)\) such that
		\begin{equation}\label{eqnmin}
			I_\ell(u_\ell)=\lambda_\ell:=\inf_{u\in S^{1,2}(M)\setminus\{0\}} I_\ell(u)<K(n)^{-2},\quad\|u_\ell\|_{L^{Q^*}(M)}=1.
		\end{equation}
		Moreover, \(u_\ell\) is smooth and positive on \(M\), and satisfies the Euler--Lagrange equation
		\begin{equation}\label{eqnEL}
			-\Delta_b u_\ell+\ell u_\ell=\lambda_\ell u_\ell^{Q^*-1}, \quad u_\ell > 0, \quad \text { on } M .
		\end{equation}
	\end{proposition}
	\begin{proof}
		We follow the proof of Theorem~6.5 in \cite{JL1987}. For \(2<q<Q^*\), define
		\[
			\lambda_{\ell,q}:=\inf\{\|\nabla_b u\|_{L^2(M)}^2+\ell\|u\|_{L^2(M)}^2: u\in S^{1,2}(M),\ \|u\|_{L^q(M)}=1\}.
		\]
		We first show that \(\lambda_{\ell,q}\) is attained by a positive function \(u_{\ell,q}\in S^{1,2}(M)\). Let \(\{\phi_k\}\) be a minimizing sequence for \(\lambda_{\ell,q}\). Replacing \(\phi_k\) by \(|\phi_k|\), we may assume that \(\phi_k\geq0\). Since \(\{\phi_k\}\) is bounded in \(S^{1,2}(M)\), after passing to a subsequence, there exists a nonnegative function \(u_{\ell,q} \in S^{1, 2}(M)\) such that \(\phi_k \rightharpoonup u_{\ell,q}\) weakly in \(S^{1, 2}(M)\) and \(\phi_k \to u_{\ell,q}\) strongly in \(L^q(M)\). In particular, \(\|u_{\ell,q}\|_{L^q}=1\).  By weak lower semicontinuity,
		\[
		\lambda_{\ell,q} \le\|\nabla_b u_{\ell,q}\|_{L^2(M)}^2+\ell\|u_{\ell,q}\|_{L^2(M)}^2\le\liminf_{k\to\infty}(\|\nabla_b\phi_k\|_{L^2(M)}^2+\ell\|\phi_k\|_{L^2(M)}^2)=\lambda_{\ell,q}.
		\]
		Hence \(u_{\ell,q}\) attains \(\lambda_{\ell,q}\). A standard variational argument gives
		\[
			-\Delta_b u_{\ell,q}+\ell u_{\ell,q}=\lambda_{\ell,q} u_{\ell,q}^{q-1}, \quad \text { on } M.
		\]
		Since \(\|u\|_{L^q(M)}\to\|u\|_{L^{Q^*}(M)}\) as \(q\to Q^*\) for every fixed \(u\in L^{Q^*}(M)\), and since \(\lambda_\ell<K(n)^{-2}\), there exists \(q_0<Q^*\) such that \(\lambda_{\ell,q} < K(n)^{-2}\) for \(q_0 < q < Q^*\). The Moser iteration in the proof of Jerison--Lee \cite[Theorem 6.5]{JL1987} then gives an \(L^\infty\)-bound independent of \(q\), that is, there exists a constant \(C_\ell>0\), independent of \(q\), such that \(\|u_{\ell,q}\|_{L^\infty(M)} < C_\ell\)  for \(q_0 < q < Q^*\). Let \(q\to Q^*\). After passing to a subsequence,  we may assume that \(u_{\ell,q} \rightharpoonup u_\ell\) weakly in \(S^{1, 2}(M)\) and \(u_{\ell,q} \to u_\ell\) strongly in \(L^2(M)\). Moreover, \(u_{\ell,q}\to u_\ell\) almost everywhere on \(M\). Since \(\{u_{\ell,q}\}\) is uniformly bounded in \(L^\infty(M)\), the dominated convergence theorem gives \(\|u_\ell\|_{L^{Q^*}}=1\). By weak lower semicontinuity and the definition of \(\lambda_\ell\), we have
		\[
		\lambda_\ell \le I_\ell(u_\ell) \le \liminf_{q\to Q^*}\lambda_{\ell,q}.
		\]
		On the other hand, using fixed test functions and letting \(q\to Q^*\) gives
		\[
		\limsup_{q\to Q^*}\lambda_{\ell,q} \leq \lambda_\ell.
		\]
		Hence \(I_\ell(u_\ell)=\lambda_\ell\), and \(u_\ell\) is a minimizer satisfying the Euler--Lagrange equation \eqref{eqnEL}.
		
		Finally, by subelliptic regularity, \(u_\ell\) is smooth. By Bony's maximum principle \cite{B1969} and connectedness, if \(u_\ell\) vanished at some point, then \(u_\ell\equiv0\) on \(M\), contradicts \(\|u_\ell\|_{L^{Q^*}}=1\). Therefore \(u_\ell>0\) on \(M\). This proves the proposition.
	\end{proof}
	
	\begin{proposition}\label{propblowup}
		As \(\ell\to\infty\), the sequence \(\{u_\ell\}\) satisfies the following blow-up estimates:
		\begin{itemize}
			\item[\textup{(1)}] \(u_\ell\rightharpoonup0\) weakly in \(S^{1,2}(M)\) and \(u_\ell\to0\) strongly in \(L^p(M)\) for every \(1\leq p<Q^*\). In particular, after passing to a subsequence, \(u_\ell\to0\) almost everywhere on \(M\);
			
			\item[\textup{(2)}]\(\|\nabla_bu_\ell\|_{L^2(M)}^2\to K(n)^{-2}\), \(\ell\|u_\ell\|_{L^2(M)}^2 \to 0\), \(\lambda_\ell\to K(n)^{-2}\);
			
			\item[\textup{(3)}] \(\max_M u_\ell\to+\infty\).
		\end{itemize}
	\end{proposition}
	
	\begin{proof}
		By the normalization \(\|u_\ell\|_{L^{Q^*}(M)}=1\) and the definition of \(\lambda_\ell\), we have
		\[
		\|\nabla_bu_\ell\|_{L^2(M)}^2 + \ell\|u_\ell\|_{L^2(M)}^2 = I_\ell(u_\ell) = \lambda_\ell < K(n)^{-2}.
		\]
		It follows that
		\[
		\|u_\ell\|_{L^2(M)}^2 \leq \frac{K(n)^{-2}}{\ell} \to 0,
		\]
		and that \(\{u_\ell\}\) is bounded in \(S^{1,2}(M)\). Since the embedding \(S^{1,2}(M)\hookrightarrow L^p(M)\) is compact for every \(1\leq p<Q^*\), every subsequence of \(\{u_\ell\}\) has a further subsequence, still denoted by \(\{u_\ell\}\), such that
		\[
		u_\ell\rightharpoonup u \quad\text{weakly in }S^{1,2}(M), \qquad u_\ell\to u \quad\text{strongly in }L^p(M),\quad 1\leq p<Q^*.
		\]
		Since \(\|u_\ell\|_{L^2(M)}\to0\), we have \(u\equiv0\). Since the subsequence was arbitrary, we have
		\[
		u_\ell\rightharpoonup0 \quad\text{in }S^{1,2}(M), \qquad u_\ell\to0 \quad\text{in }L^p(M),\quad 1\leq p<Q^*.
		\]
		The almost everywhere convergence follows after passing to a subsequence. This proves \textup{(1)}.
		
		For \textup{(2)},  since \(\{\|\nabla_bu_\ell\|_{L^2(M)}^2\}\) is bounded, we may assume that
		\[
		\|\nabla_bu_\ell\|_{L^2(M)}^2\to\theta
		\]
		passing to a subsequence. By \eqref{equShi}, using \(\|u_\ell\|_{L^{Q^*}(M)}=1\) and \(\|u_\ell\|_{L^2(M)}\to 0\), and then letting \(\ell \to \infty\), we obtain
		\[
		1 \leq (K(n)^2+\varepsilon)\theta .
		\]
		Since \(\varepsilon>0\) is arbitrary, \(\theta\geq K(n)^{-2}\). On the other hand, since
		\[
		\|\nabla_bu_\ell\|_{L^2(M)}^2 \leq \lambda_\ell < K(n)^{-2},
		\]
		we have \(\theta \leq K(n)^{-2}\). Hence \(\theta=K(n)^{-2}\). Because the convergent subsequence was arbitrary, we have
		\[
		\|\nabla_bu_\ell\|_{L^2(M)}^2\to K(n)^{-2}.
		\]
		Using the fact that \(\lambda_\ell< K(n)^{-2}\) and \(\ell\|u_\ell\|_{L^2(M)}^2=\lambda_\ell-\|\nabla_bu_\ell\|_{L^2(M)}^2\) we have
		\[
		\lambda_\ell\to K(n)^{-2} \quad\text{and}\quad \ell\|u_\ell\|_{L^2(M)}^2\to0
		\]
		as \(\ell \to \infty\). This proves \textup{(2)}.
		
		It remains to prove \textup{(3)}. Suppose, by contradiction, that \(\max_Mu_\ell\) is bounded. Then there is a constant \(C>0\) such that \(u_\ell\leq C\) on \(M\). Hence
		\[
		1 = \int u_\ell^{Q^*}\,dv_\theta \leq (\max_Mu_\ell)^{Q^*-2} \int u_\ell^2\,dv_\theta \leq C^{Q^*-2}\|u_\ell\|_{L^2(M)}^2 \to0,
		\]
		which is impossible. Therefore \(\max_Mu_\ell\to+\infty\), and the proof is complete.
	\end{proof}
	
	Let \(q_\ell\in M\) be a maximum point of \(u_\ell\), namely, \(u_\ell(q_\ell)=\max_Mu_\ell\). Define \( \mu_\ell=u_\ell(q_\ell)^{-\frac{2}{Q-2}} . \) By Proposition~\ref{propblowup}, \(\mu_\ell\to0\) as \(\ell \to +\infty\). In pseudohermitian normal coordinates centered at \(q_\ell\), define
	\[
	v_\ell(\xi) = \mu_\ell^{\frac{Q-2}{2}} u_\ell(\Psi_\ell(\delta_{\mu_\ell}\xi)), \qquad \xi\in\Omega_\ell:=B_{r_{JL}\mu_\ell^{-1}}(0)\subset\mathbb H^n .
	\]
	Let \(\theta_\ell\) be the rescaled contact form. We write \(\nabla_{b,\ell}\) and \(\Delta_{b,\ell}\) for the corresponding horizontal gradient and sublaplacian. On the standard Heisenberg group we write \(\nabla_{\mathbb H}\), \(\Delta_{\mathbb H}\), and \(d\xi=dv_{\theta_0}\). Then
	\[
		-\Delta_{b,\ell}v_\ell+\ell\mu_\ell^2v_\ell = \lambda_\ell v_\ell^{Q^*-1}, \qquad v_\ell(0)=1,\qquad 0\le v_\ell\le1 \quad\hbox{in }\Omega_\ell .
	\]
	
	\begin{proposition}\label{propvstrong}
		Let \(v_0\) be the Jerison--Lee bubble normalized by \(v_0(0)=1\). Then
		\begin{equation}\label{equvstrong}
			\int_{\Omega_\ell} ( |\nabla_{b,\ell}(v_\ell-v_0)|^2 + |v_\ell-v_0|^{Q^*} ) \,dv_{\theta_\ell} \longrightarrow0 .
		\end{equation}
		In particular, \(v_\ell\to v_0\) in \(L^{Q^*}(\Omega_\ell)\) and in energy.
	\end{proposition}
	
	\begin{proof}
		Set \(\varepsilon_\ell =  \ell \mu_\ell^2\), by Proposition~\ref{propblowup}, we have
		\[
		\varepsilon_\ell = \frac{\ell \| u_\ell \|_{L^2(M)}^2}{(\max_M u_\ell)^{Q^*-2}\| u_\ell \|_{L^2(M)}^2} \le \frac{\ell \| u_\ell \|_{L^2(M)}^2}{\| u_\ell \|_{L^{Q^*}(M)}^2} \rightarrow 0, \quad \text{as } \ell \to \infty.
		\]
		By a change of variables,
		\[
		\int_{\Omega_\ell}v_\ell^{Q^*}\,dv_{\theta_\ell}= \int_{B_{r_{JL}}(q_\ell)}u_\ell^{Q^*}\,dv_\theta \quad \text{and}\quad \int_{\Omega_\ell}|\nabla_{b,\ell}v_\ell| ^2\,dv_{\theta_\ell}=\int_{B_{r_JL}(q_\ell)}|\nabla_bu_\ell|^2\,dv_\theta.
		\]
		Consequently, by the definition of \(u_\ell\) and Proposition~\ref{propblowup}, we have
		\begin{equation}\label{equlimsupvl}
			\limsup_{\ell\to\infty}\int_{\Omega_\ell}v_\ell^{Q^*}\,dv_{\theta_\ell} \le 1, \quad \text{and}\quad \limsup_{\ell\to\infty}\int_{\Omega_\ell}|\nabla_{b,\ell}v_\ell|^2\,dv_{\theta_\ell} \le K(n)^{-2}.
		\end{equation}
		Recall that \(0 \le v_\ell \le v_\ell(0) =1\), the local subelliptic estimates of Folland--Stein \cite{FS1974} imply that, after passing to a subsequence, \(v_\ell \rightarrow v\) for some \(v\) in \(\Gamma^{1, \beta}_{loc}(\Ha)\) where \(0 < \beta <1\). Moreover, \(v\) satisfies
		\[
		\int_{\Ha} v^{Q^*}\,d \xi = \lim_{R \to +\infty}  \int_{B_R} v^{Q^*}\,d \xi = \lim_{R \to +\infty} \lim_{\ell \to +\infty} \int_{B_R} v_\ell^{Q^*}\,d V_{\theta_\ell} \le 1,
		\]
		and
		\[
		\int_{\Ha}|\nabla_{\mathbb{H}} v|^2\,d \xi = \lim_{R \to +\infty} \int_{B_R}|\nabla_{\mathbb{H}} v|^2\,d \xi= \lim_{R \to +\infty} \lim_{\ell \to +\infty} \int_{B_R}|\nabla_{b, \ell} v_\ell|^2\,d V_{\theta_\ell} \le K(n)^{-2}.
		\]
		In particular, \(v \in \mathcal{S}^{1, 2}(\Ha)\) and \(v\) is a weak solution of the equation
		\[
		-\Delta_b v = K(n)^{-2} v^{Q^*-1}, \quad 0 \le v \le 1, \quad \text{in } \Ha,
		\]
		and \(v(0) =1\) is a maximum point of \(v\). Multiplying the equation by \(v\) and integrating by parts, we have
		\[
		K(n)^{-2} \int_{\Ha} v^{Q^*}\,d \xi = \int_{\Ha}|\nabla_b v|^2\,d \xi \ge K(n)^{-2} (\int_{\Ha} v^{Q^*}\,d \xi)^{2/Q^*}.
		\]
		Therefore, \(\int_{\Ha} v^{Q^*}\,d \xi = 1\) and \(\int_{\Ha}|\nabla_b v|^2\,d \xi = K(n)^{-2}\), which shows that \(v\) is an extremal function for the sharp Sobolev inequality on \(\mathbb{H}^n\). By the classification of extremals \cite{JL1988,FL2012}, together with the normalization \(\int_{\mathbb H^n}v^{Q^*}\,dv_0=1\) and \(v(0)=\max_{\mathbb H^n}v=1\), we obtain \(v=v_0\). Since every subsequential local limit is the same, the local convergence holds for the whole sequence.
		
		We now prove the \(L^{Q^*}\)-convergence. For \(p=Q^*\), the elementary inequality
		\[
		|a+b|^p-|a|^p-|b|^p \le C_p(|a|^{p-1}|b|+|a|\,|b|^{p-1}), \quad a,b\in\mathbb R,
		\]
		applied with \(a=v_0\) and \(b=v_\ell-v_0\), gives
		\[
		\begin{aligned}
			\int_{\Omega_\ell} |v_\ell - v_0|^{Q^*}\,d v_{\theta_{\ell}} \le &\int_{\Omega_\ell} v_\ell^{Q^*} - v_0^{Q^*}\,d v_{\theta_{\ell}} + C \int_{\Omega_\ell} v_0^{Q^*-1} |v_\ell - v_0| + v_0|v_\ell - v_0|^{Q^*-1}\,d v_{\theta_{\ell}} \\
			\le &C ( \int_{\Omega_\ell\setminus B_R(0)} v_0^{Q^*}\,d v_{\theta_{\ell}} )^{\frac{Q^*-1}{Q^*}} ( \int_{\Omega_\ell\setminus B_R(0)} |v_\ell - v_0|^{Q^*}\,d v_{\theta_{\ell}} )^{\frac{1}{Q^*}} \\
			& + C ( \int_{B_R(0)} v_0^{Q^*}\,d v_{\theta_{\ell}} )^{\frac{Q^*-1}{Q^*}} ( \int_{B_R(0)} |v_\ell - v_0|^{Q^*}\,d v_{\theta_{\ell}} )^{\frac{1}{Q^*}} \\
			& + C ( \int_{\Omega_\ell\setminus B_R(0)} v_0^{Q^*}\,d v_{\theta_{\ell}} )^{\frac{1}{Q^*}} ( \int_{\Omega_\ell\setminus B_R(0)} |v_\ell - v_0|^{Q^*}\,d v_{\theta_{\ell}} )^{\frac{Q^*-1}{Q^*}} \\
			& + C ( \int_{B_R(0)} v_0^{Q^*}\,d v_{\theta_{\ell}} )^{\frac{1}{Q^*}} ( \int_{B_R(0)} |v_\ell - v_0|^{Q^*}\,d v_{\theta_{\ell}} )^{\frac{Q^*-1}{Q^*}}+o(1).
		\end{aligned}
		\]
		By taking \(R\) large enough, the first and third terms can be made arbitrarily small. By the local convergence of \(v_\ell\) to \(v_0\), the second and fourth terms tends to \(0\) as \(\ell \to \infty\). Hence \(\int_{\Omega_\ell} |v_\ell - v_0|^{Q^*}\,d v_{\theta_{\ell}} \to 0\) as \(\ell\to\infty\). The strong convergence of the gradients is straightforward:
		\[
		\begin{aligned}
			\int_{\Omega_\ell} \nabla_{b,\ell}(v_\ell-v_0)\nabla_{b,\ell}v_0\,dv_{\theta_\ell} &\le \int_{B_R} |\nabla_{b,\ell}(v_\ell-v_0)|\,|\nabla_{b,\ell}v_0| \,dv_{\theta_\ell} \\
			&\quad+ \Big( \int_{\Omega_\ell\setminus B_R} |\nabla_{b,\ell}(v_\ell-v_0)|^2\,dv_{\theta_\ell} \Big)^{1/2} \Big( \int_{\Omega_\ell\setminus B_R} |\nabla_{b,\ell}v_0|^2\,dv_{\theta_\ell} \Big)^{1/2} \\
			&\le \int_{B_R} |\nabla_{b,\ell}(v_\ell-v_0)|\,|\nabla_{b,\ell}v_0| \,dv_{\theta_\ell} + C \Big( \int_{\mathbb H^n\setminus B_R} |\nabla_{\mathbb H}v_0|^2\,d\xi \Big)^{1/2}.
		\end{aligned}
		\]
		and therefore
		\[
		\lim_{\ell\to\infty} \int_{\Omega_\ell} \nabla_{b,\ell}(v_\ell-v_0)\nabla_{b,\ell}v_0 \,dv_{\theta_\ell} =0.
		\]
		Consequently, by \eqref{equlimsupvl} and since
		\[
		\int_{\Omega_\ell} |\nabla_{b,\ell}v_0|^2\,dv_{\theta_\ell} \to \int_{\mathbb H^n} |\nabla_{\mathbb H}v_0|^2\,d\xi = K(n)^{-2},
		\]
		we conclude:
		\[
		\begin{aligned}
			\int_{\Omega_\ell} |\nabla_{b,\ell}(v_\ell-v_0)|^2\,dv_{\theta_\ell} &= \int_{\Omega_\ell} |\nabla_{b,\ell}v_\ell|^2\,dv_{\theta_\ell} - \int_{\Omega_\ell} |\nabla_{b,\ell}v_0|^2\,dv_{\theta_\ell} \\
			&\quad -2 \int_{\Omega_\ell} \nabla_{b,\ell}(v_\ell-v_0)\nabla_{b,\ell}v_0 \,dv_{\theta_\ell} \le o(1).
		\end{aligned}
		\]
		Thus the strong convergence of the gradients follows.
	\end{proof}
	
	\begin{corollary}\label{coronepoint}
		For any \(\varepsilon>0\), there exists \(\delta_\varepsilon>0\) and \(\ell_\varepsilon>0\) such that for all \(\ell>\ell_\varepsilon\),
		\[
		\int_{M\setminus B_{\mu_\ell/\delta_\varepsilon}(q_\ell)} | \nabla_b u_\ell|^2 + u_\ell^{Q^*} \,dv_\theta \le \varepsilon.
		\]
		In particular, for any fixed \(\rho>0\),
		\[
		\lim_{\ell\to\infty} \int_{M\setminus B_{\rho}(q_\ell)} | \nabla_b u_\ell|^2 + u_\ell^{Q^*} \,dv_\theta = 0.
		\]
	\end{corollary}
	\begin{proof}
		By \eqref{equvstrong}, for any \(\varepsilon>0\), after a rescaling, there exists \(\delta_\varepsilon>0\)  and \(\ell_\varepsilon^\prime>0\) such that
		\[
		\int_{B_{\mu_\ell/\delta_\varepsilon}(q_\ell)}|\nabla_b u_\ell|^2  \,dv_\theta \ge \int_{\mathbb H^n}|\nabla_{\mathbb H}v_0|^2\,d\xi - \frac{\varepsilon}{4} = K(n)^{-2} - \frac{\varepsilon}{4}.
		\]
		and
		\[
		\int_{B_{\mu_\ell/\delta_\varepsilon}(q_\ell)} u_\ell^{Q^*} \,dv_\theta \ge \int_{\mathbb H^n} v_0^{Q^*} \,d\xi - \frac{\varepsilon}{4}= 1 - \frac{\varepsilon}{4}.
		\]
		Recalling that \(\|u_\ell\|_{L^{Q^*}(M)}=1\) and \(\|\nabla_b u_\ell\|_{L^2(M)}^2 \to K(n)^{-2}\), we can choose some \(\ell_\varepsilon>\ell_\varepsilon^\prime\) such that for all \(\ell>\ell_\varepsilon\),
		\[
		\int_{M\setminus B_{\mu_\ell/\delta_\varepsilon}(q_\ell)} | \nabla_b u_\ell|^2 + u_\ell^{Q^*} \,dv_\theta \le \varepsilon.
		\]
	\end{proof}
	
	\section{A pointwise estimate}\label{sectionpointwise}
	Throughout this section, we use this notation with \(q=q_\ell\), and write
	\[
	\Psi_\ell:=\Psi_{q_\ell},\qquad \rho_\ell:=\rho_{q_\ell}.
	\]
	
	\begin{proposition}\label{propptwise}
		There exists a constant \(C>0\), independent of \(\ell\), such that for all sufficiently large \(\ell\),
		\[
			u_\ell(x) \le C\mu_\ell^{\frac{Q-2}{2}} (\mu_\ell+\rho_\ell(x))^{2-Q}, \qquad x\in B_{r_G}(q_\ell),
		\]
		where \(r_G\le r_{JL}\) is a positive constant depending only on \((M, \theta)\). Consequently, in the rescaled coordinates
		\[
		v_\ell(\xi)= \mu_\ell^{\frac{Q-2}{2}} u_\ell(\Psi_\ell(\delta_{\mu_\ell}\xi)), \qquad \xi\in\Omega_\ell:=B_{r_G\mu_\ell^{-1}}(0),
		\]
		one has
		\[
			v_\ell(\xi) \le \frac{C}{1+\rho_0(\xi)^{Q-2}}, \qquad \xi\in\Omega_\ell .
		\]
	\end{proposition}
	
	\begin{proof}
		In Appendix~\ref{appendixgreen}, we construct the Green function \(G_a\) of the operator \(L_1=-\Delta_b+1\) with pole at \(a\).  It satisfies
		\[
		L_1G_a=\delta_a
		\]
		in the sense of distributions. By Proposition~\ref{propgreenestimate}, applied with \(a=q_\ell\), there exist constants \(C>0\) and \(r_G>0\), depending only on \((M,\theta)\), such that
		\[
		C^{-1}\rho_\ell(x)^{2-Q} \le G_{q_\ell}(x) \le C\rho_\ell(x)^{2-Q}, \qquad 0<\rho_\ell(x)<r_G .
		\]
		
		Define
		\[
		\Phi_\ell(x)=\mu_\ell^{\frac{Q-2}{2}}G_{q_\ell}(x).
		\]
		Then
		\[
		(-\Delta_b+1)\Phi_\ell=0\quad\text{in }M\setminus\{q_\ell\},
		\]
		and, since \(\ell\ge1\),
		\begin{equation}\label{eq:phi-supersolution}
(-\Delta_b+\ell)\Phi_\ell =(\ell-1)\Phi_\ell\ge0 \quad\text{in }M\setminus\{q_\ell\}.
\end{equation}
		In particular, \(\Phi_\ell\) is a positive supersolution for the operator \(-\Delta_b+\ell\) away from \(q_\ell\).
		
		We shall prove that \(u_\ell/\Phi_\ell\) is uniformly bounded outside a fixed multiple of the blow-up scale. Fix \(R\) in Corollary~\ref{coronepoint} so large that the inequality holds with \(\varepsilon=\varepsilon_0\), where \(\varepsilon_0>0\) will be chosen below.
		
		Set
		\[
		D_{\ell,R}=M\setminus B_{R\mu_\ell}(q_\ell), \qquad \zeta_\ell=\frac{u_\ell}{\Phi_\ell}.
		\]
		We derive a weak differential inequality for \(\zeta_\ell\) on \(D_{\ell,R}\). Let \(\psi\in C_c^\infty(M\setminus\{q_\ell\})\), \(\psi\ge0\). Testing \eqref{eqnEL} with \(\Phi_\ell\psi\) and \eqref{eq:phi-supersolution} with \(u_\ell\psi\), we obtain
		\[
		\int_M \Phi_\ell^2 \langle\nabla_b\zeta_\ell,\nabla_b\psi\rangle\,dv_\theta \le \lambda_\ell \int_M \Phi_\ell^{Q^*}\zeta_\ell^{Q^*-1}\psi\,dv_\theta.
		\]
		This is the CR analogue of the quotient inequality used in the Green-function comparison argument of Li and Ricciardi \cite{LR2003}.
		
		We next record the weighted Sobolev inequality associated with \(\Phi_\ell\). If \(w\in C_c^\infty(M\setminus\{q_\ell\})\), then
		\begin{equation}\label{eq:weighted-sobolev}
\Big( \int_M |w|^{Q^*}\Phi_\ell^{Q^*}\,dv_\theta \Big)^{2/Q^*} \le C\int_M\Phi_\ell^2|\nabla_bw|^2\,dv_\theta.
\end{equation}
		Indeed, applying the Folland--Stein Sobolev inequality to \(\Phi_\ell w\) gives
		\[
		\|\Phi_\ell w\|_{L^{Q^*}(M)}^2 \le C\int_M ( |\nabla_b(\Phi_\ell w)|^2+\Phi_\ell^2w^2 )\,dv_\theta.
		\]
		Since \((-\Delta_b+1)\Phi_\ell=0\) on the support of \(w\), integration by parts gives
		\[
		\int_M ( |\nabla_b(\Phi_\ell w)|^2+\Phi_\ell^2w^2)\,dv_\theta = \int_M\Phi_\ell^2|\nabla_bw|^2\,dv_\theta.
		\]
		This proves the weighted Sobolev inequality.

        We now establish a uniform bound for \(\zeta_\ell\). Set
\[
\kappa=\frac{Q^*}{2}=\frac{Q}{Q-2},\qquad d\nu_\ell=\Phi_\ell^{Q^*}\,dv_\theta,
\]
and write
\[
D_{\ell,a}=M\setminus B_{a\mu_\ell}(q_\ell), \qquad V_\ell=\lambda_\ell\zeta_\ell^{Q^*-2}.
\]
For \(m\ge2\) and \(\chi\in C_c^\infty(M\setminus\{q_\ell\})\), \(0\le\chi\le1\), testing the quotient inequality with \(\psi=\chi^2\zeta_\ell^{m-1}\), and applying \eqref{eq:weighted-sobolev} yields
\begin{equation}\label{eququotient}
\|\chi\zeta_\ell^{m/2}\|_{L^{Q^*}(d\nu_\ell)}^2 \le Cm^2\Big(\int_M \Phi_\ell^{2-Q^*}|\nabla_b\chi|^2 \zeta_\ell^m\,d\nu_\ell+\int_M V_\ell (\chi\zeta_\ell^{m/2})^2\,d\nu_\ell \Big),
\end{equation}
where \(C\ge1\) is independent of \(m\), \(\ell\), and \(\chi\).

We first improve the integrability of \(\zeta_\ell\). Choose \(\varepsilon_0>0\) so small that
\[
C(Q^*)^2K(n)^{-2}\varepsilon_0^{2/Q}\le\frac12,
\]
and then choose \(R\) according to Corollary~\ref{coronepoint}, so that
\[
\int_{D_{\ell,R}}\zeta_\ell^{Q^*}\,d\nu_\ell =\int_{D_{\ell,R}}u_\ell^{Q^*}\,dv_\theta \le\varepsilon_0
\]
for all sufficiently large \(\ell\). Since \(R\) is fixed, we may also assume that \(2R\mu_\ell<r_G\). Let \(\chi\) be a smooth cutoff satisfying
\[
0\le\chi\le1,\qquad \chi=0\ \text{on }B_{R\mu_\ell}(q_\ell),\qquad \chi=1\ \text{on }D_{\ell,3R/2},\qquad |\nabla_b\chi|\le\frac{C}{R\mu_\ell}.
\]
On the support of \(\nabla_b\chi\), the local Green estimate gives
\[
\Phi_\ell^{2-Q^*}|\nabla_b\chi|^2 \le C\mu_\ell^{-2}\rho_\ell^4 \frac{1}{R^2\mu_\ell^2} \le C_R.
\]
Moreover, since \((Q^*-2)Q/2=Q^*\) and \(\lambda_\ell\le K(n)^{-2}\), H\"older's inequality implies
\[
\begin{aligned}
\int_M V_\ell(\chi\zeta_\ell^{Q^*/2})^2\,d\nu_\ell &\le K(n)^{-2} \left(\int_{D_{\ell,R}}\zeta_\ell^{Q^*}\,d\nu_\ell\right)^{2/Q} \|\chi\zeta_\ell^{Q^*/2}\|_{L^{Q^*}(d\nu_\ell)}^2\\
&\le K(n)^{-2}\varepsilon_0^{2/Q} \|\chi\zeta_\ell^{Q^*/2}\|_{L^{Q^*}(d\nu_\ell)}^2.
\end{aligned}
\]
Taking \(m=Q^*\), we obtain
\[
\|\zeta_\ell\|_{L^{Q^*\kappa}(D_{\ell,3R/2},d\nu_\ell)}^{Q^*} \le \|\chi\zeta_\ell^{Q^*/2}\|_{L^{Q^*}(d\nu_\ell)}^2 \le C_R\varepsilon_0.
\]
Consequently,
\[
\|V_\ell\|_{L^s(D_{\ell,3R/2},d\nu_\ell)}\le C_R, \qquad s=\frac{Q^*\kappa}{Q^*-2} =\frac{Q^2}{2(Q-2)}>\frac Q2.
\]

This improved integrability allows us to carry out Moser iteration. For every smooth \(f\) supported in \(D_{\ell,3R/2}\), H\"older's inequality and interpolation give
\[
\int_M V_\ell f^2\,d\nu_\ell \le C_R\|f\|_{L^{2s/(s-1)}(d\nu_\ell)}^2 \le C_R \|f\|_{L^{Q^*}(d\nu_\ell)}^{2\tau} \|f\|_{L^2(d\nu_\ell)}^{2(1-\tau)},
\]
where \(\tau=Q/(2s)=(Q-2)/Q<1\). Hence, for \(0<\delta<1\), Young's inequality yields
\[
\int_M V_\ell f^2\,d\nu_\ell \le \delta\|f\|_{L^{Q^*}(d\nu_\ell)}^2 +C_R\delta^{-(Q-2)/2} \|f\|_{L^2(d\nu_\ell)}^2.
\]

Set
\[
a_j=2R-R2^{-j-1},\qquad m_j=Q^*\kappa^{j+1},\qquad j=0,1,\ldots.
\]
Thus \(a_0=3R/2\), \(a_j\) increases to \(2R\), and \(m_{j+1}=\kappa m_j\). Choose smooth cutoffs \(\chi_j\) such that
\[
0\le\chi_j\le1,\qquad \chi_j=0\ \text{on }B_{a_j\mu_\ell}(q_\ell),\qquad \chi_j=1\ \text{on }D_{\ell,a_{j+1}},\qquad |\nabla_b\chi_j|\le\frac{C2^j}{R\mu_\ell},
\]
where the gradients are supported in the annulus \(R\mu_\ell\le\rho_\ell\le2R\mu_\ell\). There the local Green estimate implies
\[
\Phi_\ell^{2-Q^*}|\nabla_b\chi_j|^2\le C_R4^j.
\]
Substituting the estimate for the potential term, with \(f=\chi_j\zeta_\ell^{m/2}\) and \(\delta=(2Cm^2)^{-1}\), into \eqref{eququotient} and absorbing the \(L^{Q^*}\)-term, we obtain
\[
\|\chi_j\zeta_\ell^{m/2}\|_{L^{Q^*}(d\nu_\ell)}^2 \le C_Rm^Q4^j \int_{D_{\ell,a_j}}\zeta_\ell^m\,d\nu_\ell, \qquad m\ge2.
\]
Here and below, \(C_R\) denotes a positive constant independent of \(m\), \(j\), and \(\ell\). Setting \(m=m_j\), we obtain
\[\|\zeta_\ell\|_{L^{m_{j+1}}(D_{\ell,a_{j+1}},d\nu_\ell)}\le\bigl(C_Rm_j^Q4^j\bigr)^{1/m_j}\|\zeta_\ell\|_{L^{m_j}(D_{\ell,a_j},d\nu_\ell)}.\]
Since \(m_j=Q^*\kappa^{j+1}\) with \(\kappa>1\), we have
\[\sum_{j=0}^{\infty}\frac{1+j+\log m_j}{m_j}<\infty.\]
Thus the product of the iteration factors is finite. Iterating the preceding estimate and using the \(L^{Q^*\kappa}\)-bound on \(D_{\ell,3R/2}\), we obtain a uniform \(L^\infty(D_{\ell,2R},d\nu_\ell)\)-bound for \(\zeta_\ell\). Since \(\Phi_\ell>0\) and \(\zeta_\ell\) is continuous away from \(q_\ell\), this bound holds pointwise, and hence
\[\sup_{M\setminus B_{2R\mu_\ell}(q_\ell)}\zeta_\ell\le C_R.\]
Therefore,
		\[
		u_\ell(x)\le C_R\Phi_\ell(x), \qquad x\in M\setminus B_{2R\mu_\ell}(q_\ell).
		\]
		In particular, for \(x\in B_{r_G}(q_\ell)\setminus B_{2R\mu_\ell}(q_\ell)\), the local Green kernel estimate yields
		\[
		u_\ell(x) \le C_R\mu_\ell^{\frac{Q-2}{2}}\rho_\ell(x)^{2-Q}.
		\]
		
		It remains to treat the core region \(B_{2R\mu_\ell}(q_\ell)\). By the local blow-up rescaling already established,
		\[
		v_\ell(\xi) = \mu_\ell^{\frac{Q-2}{2}} u_\ell(\Psi_{q_\ell}(\delta_{\mu_\ell}\xi))
		\]
		is locally uniformly bounded. Hence, for \(\rho_\ell(x)\le2R\mu_\ell\),
		\[
		u_\ell(x) \le C_R\mu_\ell^{-\frac{Q-2}{2}} \le C_R \mu_\ell^{\frac{Q-2}{2}} (\mu_\ell+\rho_\ell(x))^{2-Q}.
		\]
		Combining this core estimate with the exterior estimate gives
		\[
		u_\ell(x) \le C \mu_\ell^{\frac{Q-2}{2}} (\mu_\ell+\rho_\ell(x))^{2-Q}, \qquad x\in B_{r_G}(q_\ell),
		\]
		after enlarging \(C\). Finally, if \(x=\Psi_{q_\ell}(\delta_{\mu_\ell}\xi)\), then \(\rho_\ell(x)=\mu_\ell\rho_0(\xi)\), and therefore
		\[
		v_\ell(\xi) \le C(1+\rho_0(\xi))^{2-Q}.
		\]
		The proposition is proved.
	\end{proof}
	
	\begin{corollary}\label{estimatering}
		For every \(0 < r \le r_G\), there exists  \(\ell_r>0\) and a constant \(C>0\), independent of \(\ell\), such that for all \(\ell>\ell_r\),
		\[
		\int_{B_{r}(q_\ell)} u_\ell^2\,dv_\theta \le C\int_{B_{r/2}(q_\ell)} u_\ell^2\,dv_\theta,
		\]
	\end{corollary}
	
	\begin{proof}
		By the pointwise estimate in Proposition~\ref{propptwise} and the local volume comparison in pseudohermitian normal coordinates, we have
		\[
		\begin{aligned}
			\int_{B_{r}(q_\ell)\setminus B_{r/2}(q_\ell)}u_\ell^2\,dv_\theta &\le C\mu_\ell^{Q-2} \int_{r/2}^{r} \frac{\rho^{Q-1}}{(\mu_\ell+\rho)^{2Q-4}}\,d\rho  \\
			&\le C\mu_\ell^{Q-2}.
		\end{aligned}
		\]
		On the other hand, by the rescaling centered at \(q_\ell\), for all sufficiently large \(\ell\),
		\[
		\int_{B_{r/2}(q_\ell)}u_\ell^2\,dv_\theta \ge \int_{B_{\mu_\ell}(q_\ell)}u_\ell^2\,dv_\theta = \mu_\ell^2 \int_{B_1(0)}v_\ell^2\,dv_{\theta_\ell} \ge C\mu_\ell^2 .
		\]
		Here the last inequality follows from the local convergence \(v_\ell\to v_0\) and \(v_0(0)=1\). Since \(Q\ge4\) and \(\mu_\ell\to0\), we have
		\[
		\int_{B_{r}(q_\ell)\setminus B_{r/2}(q_\ell)}u_\ell^2\,dv_\theta \le C\mu_\ell^{Q-2} \le C\mu_\ell^2 \le C\int_{B_{r/2}(q_\ell)}u_\ell^2\,dv_\theta .
		\]
		Therefore,
		\[
		\int_{B_{r}(q_\ell)}u_\ell^2\,dv_\theta = \int_{B_{r/2}(q_\ell)}u_\ell^2\,dv_\theta + \int_{B_{r}(q_\ell)\setminus B_{r/2}(q_\ell)}u_\ell^2\,dv_\theta \le C\int_{B_{r/2}(q_\ell)}u_\ell^2\,dv_\theta .
		\]
		This proves the corollary.
	\end{proof}

	\section{The proof of Theorem~\ref{Mainthm}}\label{sectionpohozaev}

Choose \(r_P>0\) sufficiently small so that \( 2r_P<r_G, \) and so that the Jerison--Lee coordinate estimates of Section~\ref{sectionpreliminary} hold uniformly on \(B_{2r_P}(q)\) for every \(q\in M\).  Throughout this section, unless otherwise specified, all integrals with respect to \(dv_\theta\) are taken over \(B_{r_P}(q_\ell)\).  We write
\[
\Psi_\ell:=\Psi_{q_\ell}, \qquad \rho:=\rho_\ell.
\]

Choose \(\eta_\ell\in C^\infty(M)\) such that
\[
0\le\eta_\ell\le1, \qquad \eta_\ell\equiv1 \quad\text{on }B_{r_P/2}(q_\ell), \qquad \operatorname{supp}\eta_\ell\Subset B_{r_P}(q_\ell).
\]
Since \(r_P\) is fixed and \(M\) is compact, the functions \(\eta_\ell\) may be chosen so that all their derivatives used below are bounded uniformly in \(\ell\).  Set \( w_\ell:=\eta_\ell u_\ell. \)

We first record a uniform lower bound for the critical norm of \(w_\ell\). Since \(\mu_\ell\to0\), for all sufficiently large \(\ell\),
\[
B_{\mu_\ell}(q_\ell)\subset B_{r_P/2}(q_\ell).
\]
Consequently,
\[
\int w_\ell^{Q^*}\,dv_\theta \ge \int_{B_{\mu_\ell}(q_\ell)} u_\ell^{Q^*}\,dv_\theta = \int_{B_1(0)} v_\ell^{Q^*}\,dv_{\theta_\ell}.
\]
By the local convergence \(v_\ell\to v_0\), there exists \(c_0>0\), independent of \(\ell\), such that
\begin{equation}\label{eq:cutoff-mass-lower}
\int w_\ell^{Q^*}\,dv_\theta\ge c_0
\end{equation}
for all sufficiently large \(\ell\).

The following estimate shows how the error term of Jerison--Lee normal coordinates controls \(\ell \int w_\ell^2\,dv_\theta\).  Its proof compares the localized quotient directly with the sharp Folland--Stein inequality on \(\mathbb H^n\).

\begin{lemma}\label{lemmapohozaev}
For all sufficiently large \(\ell\),
\[
\ell\int w_\ell^2\,dv_\theta \le C\int \rho^2w_\ell^{Q^*}\,dv_\theta + C\int \rho^2|\nabla_bw_\ell|^2\,dv_\theta + C\int \rho^4|T_\theta w_\ell|^2\,dv_\theta,
\]
where \(C>0\) is independent of \(\ell\).
\end{lemma}

\begin{proof}
Multiplying equation \eqref{eqnEL} by \(\eta_\ell^2u_\ell\) and integrating by parts, we obtain
\begin{equation}\label{eq:cutoff-energy-identity}
\int|\nabla_bw_\ell|^2\,dv_\theta + \ell\int w_\ell^2\,dv_\theta= \lambda_\ell \int\eta_\ell^2u_\ell^{Q^*}\,dv_\theta \\
+ \int|\nabla_b\eta_\ell|^2u_\ell^2\,dv_\theta.
\end{equation}

By H\"older's inequality and the normalization of \(u_\ell\) in \eqref{eqnmin},
\[
\int\eta_\ell^2u_\ell^{Q^*}\,dv_\theta \le \Big(\int w_\ell^{Q^*}\,dv_\theta\Big)^{2/Q^*} \Big(\int_Mu_\ell^{Q^*}\,dv_\theta\Big)^{(Q^*-2)/Q^*} = \Big(\int w_\ell^{Q^*}\,dv_\theta\Big)^{2/Q^*}.
\]
Moreover, using the uniform bound for \(\nabla_b\eta_\ell\) and applying Corollary~\ref{estimatering} with the fixed radius \(r_P\), we have
\[
\int|\nabla_b\eta_\ell|^2u_\ell^2\,dv_\theta \le C\int_{B_{r_P}(q_\ell)}u_\ell^2\,dv_\theta \le C\int_{B_{r_P/2}(q_\ell)}u_\ell^2\,dv_\theta \le C\int w_\ell^2\,dv_\theta \le \frac{\ell}{2}\int w_\ell^2\,dv_\theta
\]
when \(\ell\) is sufficiently large. It follows from \eqref{eq:cutoff-energy-identity} that
\begin{equation}\label{eq:cutoff-below-sharp}
\int|\nabla_bw_\ell|^2\,dv_\theta + \frac{\ell}{2}\int w_\ell^2\,dv_\theta \le \lambda_\ell \Big(\int w_\ell^{Q^*}\,dv_\theta\Big)^{2/Q^*}\le K(n)^{-2} \Big(\int w_\ell^{Q^*}\,dv_\theta\Big)^{2/Q^*}.
\end{equation}

We now compare the right-hand side of \eqref{eq:cutoff-below-sharp} with the sharp inequality on the Heisenberg group.  Define
\[
\widetilde w_\ell(\xi) := w_\ell(\Psi_\ell(\xi)), \qquad \xi\in B_{r_P}(0),
\]
and extend \(\widetilde w_\ell\) by zero outside \(B_{r_P}(0)\). Since \(w_\ell\) vanishes near \(\partial B_{r_P}(q_\ell)\), this extension belongs to \(S^{1,2}(\mathbb H^n)\).

By the inverse form of the Jerison--Lee frame expansion \eqref{eq:JL-frame-comparison}, together with Young's inequality, we have
\[
|\nabla_H\widetilde w_\ell(\xi)|^2 \le (1+C\rho(\Psi_\ell(\xi))^2) |\nabla_bw_\ell|^2(\Psi_\ell(\xi)) + C\rho(\Psi_\ell(\xi))^4 |T_\theta w_\ell|^2(\Psi_\ell(\xi)).
\]
Moreover, the inverse form of \eqref{eq:JL-volume-comparison} gives \( dv_0=(1+O(\rho^2))\Psi_\ell^*dv_\theta \) and therefore yields
\begin{equation}\label{eqexpnablab}
\int_{\mathbb H^n}|\nabla_H\widetilde w_\ell|^2\,dv_0 \le \int |\nabla_bw_\ell|^2\,dv_\theta + C\int \rho^2|\nabla_bw_\ell|^2\,dv_\theta + C\int \rho^4|T_\theta w_\ell|^2\,dv_\theta .
\end{equation}
By the inverse volume comparison and \eqref{eq:cutoff-mass-lower}, the mean value theorem gives
\begin{equation}\label{eqexpqstar}
\Big(\int w_\ell^{Q^*}\,dv_\theta\Big)^{2/Q^*} \le \Big(\int_{\mathbb H^n} \widetilde w_\ell^{Q^*}\,dv_0\Big)^{2/Q^*}+ C\int\rho^2w_\ell^{Q^*}\,dv_\theta.
\end{equation}
Combining \eqref{eq:cutoff-below-sharp}, \eqref{eqexpnablab}, and \eqref{eqexpqstar} with the sharp Folland--Stein inequality on \(\mathbb H^n\), we obtain
\[
\begin{aligned}
\int |\nabla_bw_\ell|^2\,dv_\theta +\frac{\ell}{2}\int w_\ell^2\,dv_\theta &\le K(n)^{-2} \Big(\int w_\ell^{Q^*}\,dv_\theta\Big)^{2/Q^*} \\
&\le K(n)^{-2} \Big(\int_{\mathbb H^n}\widetilde w_\ell^{Q^*}\,dv_0\Big)^{2/Q^*} + C\int\rho^2w_\ell^{Q^*}\,dv_\theta \\
&\le \int_{\mathbb H^n}|\nabla_H\widetilde w_\ell|^2\,dv_0 + C\int\rho^2w_\ell^{Q^*}\,dv_\theta \\
&\le \int |\nabla_bw_\ell|^2\,dv_\theta + C\int\rho^2|\nabla_bw_\ell|^2\,dv_\theta \\
&\quad + C\int\rho^4|T_\theta w_\ell|^2\,dv_\theta + C\int\rho^2w_\ell^{Q^*}\,dv_\theta .
\end{aligned}
\]
This proves the lemma.
\end{proof}
	\begin{proof}[Proof of Theorem~\ref{Mainthm}]
		Now we estimate the right-hand side in Lemma~\ref{lemmapohozaev}. Firstly, by Proposition~\ref{propptwise}, we have
		\begin{equation}\label{estLQ}
			\int \rho^2 w_\ell^{Q^*}\,dv_\theta \le C \mu_\ell^2\int_{\mathbb{H}^n} \rho^2 v_0^{Q^*}\,dv_0 \le C\mu_\ell^2.
		\end{equation}
		Multiplying the equation satisfied by \(u_\ell\) by \(\rho^2 \eta_\ell^2 u_\ell\) and integrating by parts, we have
		\[
		\begin{aligned}
			\lambda_\ell \int \rho^2 \eta_\ell^2 u_\ell^{Q^*}\,dv_\theta =& \int \rho^2 \eta_\ell^2 |\nabla_b u_\ell|^2\,dv_\theta + 2\int \rho \eta_\ell^2 u_\ell  \nabla_b \rho \nabla_b u_\ell\,dv_\theta + 2\int \rho^2 \eta_\ell u_\ell  \nabla_b \eta_\ell \nabla_b u_\ell\,dv_\theta\\
			& + \ell \int \rho^2 \eta_\ell^2 u_\ell^2\,dv_\theta\\
			\ge & \frac{1}{2}\int \rho^2|\nabla_b w_\ell|^2\,dv_\theta - C\int \rho^2 |\nabla_b \eta_\ell|^2 u_\ell^2\,dv_\theta - C\int |\nabla_b \rho|^2 \eta_\ell^2 u_\ell^2\,dv_\theta + \ell \int \rho^2 w_\ell^2\,dv_\theta\\
			\ge &\frac{1}{2}\int \rho^2|\nabla_b w_\ell|^2\,dv_\theta - C\int w_\ell^2\,dv_\theta.
		\end{aligned}
		\]
		Together with \eqref{estLQ}, we have
		\begin{equation}\label{estgrad}
			\int \rho^2 |\nabla_b w_\ell|^2\,dv_\theta \le C\mu_\ell^2 + C\int w_\ell^2\,dv_\theta.
		\end{equation}
		Similarly we have
		\[
			\int \rho^2 \eta_\ell^2 |\nabla_b u_\ell|^2\,dv_\theta \le C\mu_\ell^2  + C\int w_\ell^2\,dv_\theta.
		\]
		We next estimate the term involving \(T_\theta w_\ell\). We use the Folland--Stein subelliptic estimate in the form stated by Jerison--Lee \cite[Proposition 5.7(c)]{JL1987}. Since the Reeb direction \(T_\theta\) has weighted order two in the Folland--Stein calculus, this estimate implies that, for every \(q\in M\) and every \(\varphi\in C_c^\infty(B_{2r_P}(q))\),
		\[
		\|T_\theta\varphi\|_{L^2(M)} \leq C( \|-\Delta_b\varphi\|_{L^2(M)} +\|\varphi\|_{L^2(M)}),
		\]
		where \(C\) is independent of \(q\). The uniformity follows from the compactness of \(M\) and the uniform control of the local pseudohermitian frames. Moreover, integrating by part we have
		\[
		\int_M (-\Delta_b \varphi + \ell \varphi)^2\,d v_\theta = \int_M (-\Delta_b \varphi)^2 + 2 \ell |\nabla_b \varphi|^2 +  \ell^2\varphi^2\,dv_\theta \ge \int_M (-\Delta_b \varphi)^2\,d v_\theta
		\]
		Combining the preceding two estimates, we obtain
		\[
			\|T_\theta\varphi\|_{L^2(M)} \leq C( \|(-\Delta_b+ \ell)\varphi\|_{L^2(M)} +\|\varphi\|_{L^2(M)}),
		\]
		
		Now we apply  Caccioppoli estimate to \(\int \rho^4|Tu_\ell|^2\). Firstly, for \(R_0 > 0\), using Arzela--Ascoli theorem, we have
		\[
		v_\ell \to v_0 \quad \text{in } \Gamma^{2,\beta}(B_{R_0}(0))
		\]
		in the sense of subsequences. As a result, we have
		\begin{equation}\label{estTcentre}
			\int_{B_{R_0\mu}(q_\ell)} \rho^4|T_\theta w_\ell|^2\,dv_\theta \le C \mu_\ell^2 \int_{B_{R_0}(0)} \rho^4|T v_0|^2\,dv_0 \le C\mu_\ell^2.
		\end{equation}
		Let \(K_\ell\) be the integer determined by
		\[
		2^{K_\ell}\mu_\ell R_0\le r_P<2^{K_\ell+1}\mu_\ell R_0 .
		\]
		For \(0\le k\le K_\ell\), we denote \(U_k = B_{2^{k+1}\mu_{\ell } R_0}(q_\ell)\setminus B_{2^{k}\mu_{\ell} R_0}(q_\ell)\) and \(U_k^\prime = B_{2^{k+2}\mu_{\ell} R_0}(q_{\ell})\setminus B_{2^{k-1}\mu_{\ell} R_0}(q_{\ell})\). Let \(\chi_k\) be some smooth cut-off functions supported in \(U_k^\prime\) and \(\chi_k = 1\) on \(U_k\). Moreover, we can choose \(\chi_k\) such that \(|\nabla_b \chi_k| \le C(2^k\mu_{\ell} R_0)^{-1}\) and \(\Delta_b \chi_k \le C(2^{k}\mu_{\ell} R_0)^{-2}\). Then we have
		\[
			\begin{aligned}
				\int_{{U_k}(q_{\ell})} \rho^4|T_\theta w_{\ell}|^2\,dv_\theta \le &\int_{U_k^\prime} (2^{k+1}\mu_{\ell} R_0)^4 | T_\theta (\chi_k w_{\ell})|^2\,dv_\theta\\
				\le &C (2^{k+1}\mu_{\ell} R_0)^4\int_{U_k^\prime} |- \Delta_b (\chi_k w_\ell) + \ell (\chi_k w_\ell)|^2 + (\chi_k w_\ell)^2\,dv_\theta\\
				\le &C (2^{k+1}\mu_{\ell} R_0)^4\int_{U_k^\prime} \lambda_\ell^2 u_\ell^{2Q^*-2} + (\Delta_b \chi_k)^2 u_\ell^2 + |\nabla_b \chi_k|^2 |\nabla_b u_\ell|^2\\
				& \qquad \qquad \qquad+ |\nabla_b \chi_k|^2u_\ell^2 + |\nabla_b u_\ell|^2 + u_\ell^2\,dv_\theta\\
				\le &C \int_{U_k^\prime} \rho^4 \lambda_\ell^2 u_\ell^{2Q^*-2} + \rho^2|\nabla_b u_\ell|^2 + u_\ell\,dv_\theta\\
			\end{aligned}
		\]
		Sum \(k\), together with \eqref{estgrad} we have
		\[
			\int_{M\setminus B_{R_0\mu_\ell}(q_\ell)} \rho^4|T_\theta w_\ell|^2\,dv_\theta \le C \int \rho^4 \lambda_\ell^2 u_\ell^{2Q^*-2} + \rho^2|\nabla_b u_\ell|^2 + u_\ell^2\,dv_\theta \le C\mu_\ell^2 + C \int w_\ell^2\,dv_\theta.
		\]
		Combining this with \eqref{estLQ}, \eqref{estgrad} and \eqref{estTcentre} we have
		\[
		\ell \int w_\ell^2\,dv_\theta \le C \mu_\ell^2 + C \int w_\ell^2\,dv_\theta.
		\]
		For all sufficiently large \(\ell\), the last term can be absorbed into the left-hand side; hence
		\[
		\ell\int w_\ell^2\,dv_\theta\le C\mu_\ell^2 .
		\]
		On the other hand, by the blow--up convergence and \(v_0(0)=1\),
		\[
		\int w_\ell^2\,dv_\theta \ge \int_{B_{\mu_\ell}(q_\ell)}u_\ell^2\,dv_\theta = \mu_\ell^2\int_{B_1(0)}v_\ell^2\,dv_{\theta_\ell} \ge c\mu_\ell^2 .
		\]
		Thus
		\[
		c\ell\mu_\ell^2 \le \ell\int w_\ell^2\,dv_\theta \le C\mu_\ell^2,
		\]
		which is impossible as \(\ell\to\infty\). This contradiction proves Theorem~\ref{Mainthm}.
	\end{proof}
	
	\appendix
	\section{The local estimate of the Green function}\label{appendixgreen}
	Throughout this appendix, \((M,\theta)\) is a compact strictly pseudoconvex pseudohermitian CR manifold of dimension \(2n+1\)  for any \(n\geq 1\). The purpose of this appendix is to record the local two-sided estimate for the Green kernel of the fixed coercive operator
	\[
	L_1=-\Delta_b+1.
	\]
	The proof follows the same philosophy as the Green-function comparison argument used by Li--Ricciardi~\cite{LR2003}, but the required singular estimate for the Green kernel is obtained from the heat kernel estimates of Jerison--Sanchez-Calle~\cite{JS1986} together with the ball-box theorem of Nagel--Stein--Wainger~\cite{NSW1985}.
	
	We first recall the definition of the Carnot--Carath\'eodory distance used in the heat kernel estimates.
	
	\begin{definition}
		Let \((M,\theta)\) be a pseudohermitian manifold. A piecewise smooth curve \(\gamma:[0,1]\to M\) is called horizontal if \(\dot{\gamma}(t)\in H_{\gamma(t)}\) for all \(t\in[0,1]\). The Carnot--Carath\'eodory distance between two points \(p,q\in M\) is defined by
		\[
		d_{\theta}(p,q)=\inf\Big\{\int_0^1|\dot{\gamma}(t)|_\theta\,dt:\gamma(0)=p,\gamma(1)=q,\gamma\text{ is horizontal}\Big\}.
		\]
		where \(|\dot{\gamma}(t)|_\theta^2=g_\theta(\dot{\gamma}(t),\dot{\gamma}(t))\), and \(g_\theta(X,Y)=d\theta(X,JY)\) is the Levi metric on \(H\).
	\end{definition}
	
	For \(a\in M\), let \(\Psi_a\) be the pseudohermitian normal coordinate map centered at \(a\).  We write
	\[
	\rho_a(x):=\rho_0(\Psi_a^{-1}(x)), \qquad x\in B_{r_{\rm JL}}(a).
	\]
	Thus the subscript of \(\rho_a\) always denotes the center, or pole, and the argument denotes the variable point.
	
	By the Ball--Box theorem of Nagel--Stein--Wainger~\cite{NSW1985}, the Carnot--Carath\'eodory distance is locally equivalent to the pseudohermitian distance \(\rho\) defined in Section~\ref{sectionblowup}. More precisely, there exist constants \(A_0\ge1\), \(C_B\ge1\), and \(r_{Box}\in(0,r_{JL})\), depending only on \((M,\theta)\), such that for every \(a\in M\) and every \(x\in B_{r_{Box}}(a)\),
	\begin{equation}\label{equbox}
		A_0^{-1}\rho_a(x)\le d_\theta(a, x)\le A_0\rho_a(x),
	\end{equation}
	and
	\[
	C_B^{-1} r^Q\le |B_\theta (a, r)| \le C_B r^Q, \quad 0 < r < r_{Box},
	\]
	where \(B_\theta(a,r) = \{x \in M: d_\theta(a,x) < r\}\). The constants above are uniform for all \(a\in M\) since \(M\) is compact.
	
	Now we turn to the heat kernel estimates of Jerison--S\'anchez-Calle~\cite{JS1986}. Although the original statement of the heat kernel estimates is stated for sums of squares of H\"ormander vector fields, the same estimate applies to the pseudohermitian sub-Laplacian \(\Delta_b\). We spell out the point needed for this application.
	
	Recall that in a local orthonormal horizontal frame, the pseudohermitian sub-Laplacian \(\Delta_b\) is a H\"ormander operator of sum-of-squares type modulo smooth first-order terms.  The final remark in~\cite{JS1986} states that the same bounds hold for such lower-order perturbations, provided the corresponding basic \(L^2\) energy estimate used in their proof is available. For \(\Delta_b\), this estimate follows directly from formal self-adjointness. Indeed, if \(w\in C_c^\infty(\R_\sigma\times M)\), then
	\[
	(\partial_\sigma w, \Delta_b w)_{L^2(\mathbb{R} \times M)}  =-\int_{\mathbb{R}} \int_M\langle\nabla_b \partial_\sigma w, \nabla_b w\rangle_\theta\,d V_\theta\,d \sigma =-\frac{1}{2} \int_{\mathbb{R}} \frac{d}{d \sigma}\|\nabla_b w(\sigma, \cdot)\|_{L^2(M)}^2\,d \sigma=0.
	\]
	Consequently, if \(P=\partial_\sigma-\Delta_b\), then
	\[
	\|\partial_\sigma w\|_{L^2(\mathbb{R} \times M)}^2=(\partial_\sigma w, P w)_{L^2(\mathbb{R} \times M)} \leq\|\partial_\sigma w\|_2\|P w\|_2,
	\]
	and hence
	\[
	\|\partial_\sigma w\|_2 \leq\|P w\|_2 .
	\]
	This is the exact estimate needed in the final remark of~\cite{JS1986}. Therefore, the following  local estimate of heat kernel holds for the pseudohermitian sub-Laplacian \(\Delta_b\): Let \(H(\sigma ; x, z)\) be the heat kernel of \(\Delta_b\), that is,
	\[
	e^{\sigma \Delta_b} f(x)=\int_M H(\sigma ; x, z) f(z)\,d V_\theta(z), \quad \sigma>0
	\]
	By the preceding discussion and the estimates in \cite{JS1986}, there exists some constants \(C_1, C_2, C_3, C_4>0\), \(r_{JS}>0\) and \(\sigma_{JS}>0\) such that the heat kernel satisfies the following local estimates. If \(d_\theta(x, z) < r_{JS}\), and \(0<\sigma<\sigma_{JS}\), then
	\[
	H(\sigma ; x, z) \leq \frac{C_1}{|B_\theta(x, \sqrt{\sigma})|} e^{-C_2 \frac{d_\theta(x, z)^2}{\sigma}},
	\]
	and if \(d_\theta(x, z)^2 \leq \sigma \leq 2d_\theta(x, z)^2\), then
	\[
	H(\sigma ; x, z) \geq \frac{C_3}{|B_\theta(x, \sqrt{\sigma})|} e^{-C_4 \frac{d_\theta(x, z)^2}{\sigma}},
	\]
	provided \(d_\theta(x, z) < r_{JS}\) and \(0<2d_\theta(x, z)^2 <\sigma_{JS}\).
	
	Set \(\sigma_G = \min\{\sigma_{JS}, r_{Box}^2\}\). We now choose a \(r_G > 0\) satisfing
	\[
	r_G < r_{Box}, \quad A_0 r_G < r_{JS}, \quad 2A_0^2 r_G^2 < \sigma_{G}, \quad \sqrt{2} A_0 r_G < r_{Box}.
	\]
	Then, whenever \(0<\rho_a(x)<r_G\) and \(d = d_\theta(a, x)\), one has
	\[
	d < r_{JS}, \quad 2d^2 <\sigma_{G}, \quad \sqrt{\sigma} \le r_{Box}, \text{ for every } d^2 \leq \sigma \leq 2d^2.
	\]
	Thus the lower bound and the ball-box volume estimate can be used simultaneously on the time interval \([d^2, 2d^2]\). The following estimate follows directly:
	\begin{lemma}\label{lemheatjl}
		Let \(H(\sigma ; x, z)\) be the heat kernel of \(\Delta_b\)  Then there exist constants \(C_1, C_2, C_3, C_4>0\), \(r_{G}>0\) and \(\sigma_{G}>0\), depending only on \((M, \theta)\), such that, for every \(a \in M\) and every \(x \in B_{r_G}(a) \setminus \{a\}\), we have
		\[
		H(\sigma ; x, a) \leq \frac{C_1}{\sigma^{Q/2}} e^{-C_2 \frac{\rho_a(x)^2}{\sigma}}, \quad 0<\sigma<\sigma_{G}.
		\]
		Moreover if \(d= d_\theta(a, x)\), then
		\[
		H(\sigma ; x, a) \ge \frac{C_3}{\sigma^{Q/2}} e^{-C_4 \frac{d^2}{\sigma}}, \quad d^2 \leq \sigma \leq 2d^2.
		\]
	\end{lemma}
	We now verify that the kernel obtained by integrating the heat kernel is indeed the Green kernel of the operator
	\[
	L_1=-\Delta_b+1.
	\]
	For each \(a\in M\), define
	\[
	G_a(x) := \int_0^\infty e^{-\sigma}H(\sigma;x,a)\,d\sigma, \quad x\ne a.
	\]
	Equivalently, we write \(G(a,x):=G_a(x)\).
	\begin{lemma}\label{lemgreen}
		For each \(a\in M\), the function \(G_a\) is the Green kernel of \(L_1=-\Delta_b+1\) with pole at \(a\).  More precisely,
		\[
		L_{1}G_a(x)=\delta_a(x)
		\]
		in the sense of distributions, and for every \(\varphi\in C^\infty(M)\),
		\[
		\int_M G_a(x)L_1\varphi(x)\,dv_\theta(x)=\varphi(a).
		\]
		Moreover, \(G_a(x)>0\) for \(x\ne a\), and \(G_a(x)=G_x(a)\).
	\end{lemma}
	\begin{proof}
		Let \(P_\sigma = e^{\sigma \Delta_b}\) be the heat semigroup. Since \(\Delta_b\) is self-adjoint and non-positive on \(L^2(M)\), the operator \(L_1 = -\Delta_b + 1\) is strictly positive. The quadratic form associated with \(L_1\) has domain \(S^{1,2}(M)\) and is given by
		\[
		B_1(u,u)=\int_M(|\nabla_bu|^2+u^2)\,dv_\theta.
		\]
		In particular, \(B_1\) is coercive on \(S^{1,2}(M)\), and \(L_1\) is invertible on \(L^2(M)\).
		
		We first verify the distributional identity. Let \(\varphi \in C^\infty(M)\). Then we have
		\[
		\begin{aligned}
			\int_M G_a(x) L_1 \varphi(x) \,d V_\theta(x) =& \int_0^{+\infty} e^{-\sigma} \int_M H(\sigma; x, a) L_1 \varphi(x)\,d V_\theta(x) \,d\sigma\\
			=& \int_0^{+\infty} e^{-\sigma} P_\sigma L_1 \varphi(a) \,d\sigma\\
			=& \int_0^{+\infty} e^{-\sigma} L_1 P_\sigma \varphi(a) \,d\sigma\\
			=& -\int_0^{+\infty} \frac{d}{d\sigma}(e^{-\sigma}  P_\sigma \varphi(a)) \,d\sigma\\
			=& \varphi(a) - \lim_{\sigma \to +\infty} e^{-\sigma} P_\sigma \varphi(a).
		\end{aligned}
		\]
		Since \(M\) is compact and \(P_\sigma\) is \(L^{\infty}\)-contractive, we have
		\[
		|e^{-\sigma} P_\sigma \varphi(a)| \leq e^{-\sigma} \|\varphi\|_{L^\infty(M)} \to 0 \quad \text{as } \sigma \to +\infty.
		\]
		Therefore, the distributional identity holds.
		
		The symmetry follows from the self-adjointness of the heat semigroup:
		\[
		H(\sigma; x, a) = H(\sigma; a, x)
		\]
		and hence
		\[
		G_a(x)= G_x(a).
		\]
		The positivity follows from the positivity of the heat kernel. Finally, away from the diagonal, the integral defining \(G\) is smooth by the hypoellipticity of the heat operator and the local heat kernel estimates. This completes the proof.
	\end{proof}
	
	\begin{proposition}\label{propgreenestimate}
		There exists  constants \(C>0\) and \(r_G>0\), depending only on \((M, \theta)\), such that for every \(a \in M\) and \(x \in B_{r_G}(a)\setminus\{a\}\),
		\[
		C^{-1} \rho_a(x)^{2-Q} \le G_a(x) \le C \rho_a(x)^{2-Q}.
		\]
	\end{proposition}
	\begin{proof}
		For the upper bound, split the Green integral as
		\[
		G_a(x) \le C_1\int_0^{\sigma_G} e^{-\sigma} H(\sigma; a, x) \,d\sigma + \int_{\sigma_G}^{+\infty} e^{-\sigma} H(\sigma; a, x) \,d\sigma = I_1 + I_2.
		\]
		For \(I_1\), with the change of variables \(s = \rho_a(x)^2/\sigma\), Lemma~\ref{lemheatjl} implies
		\[
		I_1 \le  C\int_0^{\sigma_G} \sigma^{-\frac{Q}{2}} e^{-C_2 \frac{\rho_a(x)^2}{\sigma}}\,d\sigma = C \rho_a(x)^{2-Q} \int_{\rho_a(x)^2/\sigma_G}^{+\infty} s^{\frac{Q}{2}-2} e^{-C_2 s} \,ds \le C \rho_a(x)^{2-Q}.
		\]
		For \(I_2\), since \(\sigma_G\) is fixed, the semigroup property gives, for \(\sigma \ge \sigma_G\),
		\[
		H(\sigma; a, x) = \int_M H(\sigma - \sigma_G/2; q, z) H(\sigma_G/2; z, x) \,d V_\theta(z) \le \sup_{z, x \in M} H(\sigma_G/2; z, x).
		\]
		The latter supremum is finite by hypoellipticity and compactness. Therefore
		\[
		I_2 \le C \int_{\sigma_G}^{+\infty} e^{-\sigma} \,d\sigma \le C.
		\]
		Since \(Q>2\), the constant term can be absorbed by the singular term, thus
		\[
		G_a(x) \le C \rho_a(x)^{2-Q}.
		\]
		
		For the lower bound, by the choice of \(r_G\), the interval \([d^2,2d^2]\) lies in the small-time range where Lemma \ref{lemheatjl} holds.  Hence, using the positivity of the heat kernel,
		\[
		G_a(x) \ge \int_{d^2}^{2d^2} e^{-\sigma} H(\sigma; a, x) \,d\sigma \ge C_3 \int_{d^2}^{2d^2} \sigma^{-Q/2} e^{-C\frac{d^2}{\sigma}} \,d\sigma \ge C_3 \int_{d^2}^{2d^2} \sigma^{-Q/2} \,d\sigma \ge C^{-1} d^{2-Q}.
		\]
		By the local equivalence \eqref{equbox}, we obtain
		\[
		G_a(x) \ge C^{-1} \rho_a(x)^{2-Q}.
		\]
		Thus the two-sided estimate is established.
	\end{proof}
	
		\noindent\textbf{Declarations.}~ \textbf{Conflict of interest.} ~ On behalf of all authors, the corresponding author states that there is no conflict of interest.
	

	\noindent\textbf{AI assistance statement.} ~OpenAI models were used for language editing, consistency checks, and assistance in drafting mathematical revisions and proof arguments. The authors are responsible for independently verifying the mathematical content and for the final version submitted for publication.

\end{document}